\documentclass[12pt]{amsart}
\usepackage{amsthm,amsmath,amsfonts,amssymb}
\usepackage{tikz,tikz-cd}
\usepackage{t1enc}
\usepackage[mathscr]{eucal}
\usepackage{indentfirst}
\usepackage{xcolor}
\usepackage{hyperref}
\numberwithin{equation}{section}
\usepackage[margin=2.9cm]{geometry}
\usepackage{cite}
\usepackage{parskip}
\usepackage{array}
\usepackage{booktabs}
\usepackage{enumerate}

\theoremstyle{plain}
\newtheorem{thm}{Theorem}[section]
\newtheorem{lem}[thm]{Lemma}

\newtheorem{prop}[thm]{Proposition}

 \theoremstyle{definition}
\newtheorem{defn}[thm]{Definition}
\newtheorem{rem}[thm]{Remark}
\newtheorem{?}[thm]{Problem}
\newtheorem{ex}[thm]{Example}

\newcommand{\op}[1]{\operatorname{#1}}
\newcommand{\mb}[1]{\mathbb{#1}}
\newcommand{\mc}[1]{\mathcal{#1}}
\newcommand{\mf}[1]{\mathfrak{#1}}
\newcommand{\mr}[1]{\mathrm{#1}}
\newcommand{\eps}{\varepsilon}
\newcommand{\vphi}{\varphi}

\newcommand{\Gal}{\operatorname{Gal}}
\newcommand{\Char}{\operatorname{char}}

\newcommand{\inj}{\hookrightarrow}

\newcommand{\GW}{\operatorname{GW}}
\newcommand{\W}{\operatorname{W}}
\newcommand{\Proj}{\operatorname{Proj}}
\newcommand{\Tr}{\operatorname{Tr}}

\newcommand{\sign}{\operatorname{sign}}
\newcommand{\rank}{\operatorname{rank}}
\newcommand{\disc}{\operatorname{disc}}
\newcommand{\norm}{\operatorname{norm}}
\newcommand{\Jac}{\operatorname{Jac}}
\newcommand{\codim}{\operatorname{codim}}
\newcommand{\Hom}{\operatorname{Hom}}

\newcommand{\pcoor}[1]{%
  \begingroup\lccode`~=`: \lowercase{\endgroup
  \edef~}{\mathbin{\mathchar\the\mathcode`:}\nobreak}%
  [
  \begingroup
  \mathcode`:=\string"8000
  #1%
  \endgroup 
  ]
}

\DeclareFontFamily{U}{mathx}{\hyphenchar\font45}
\DeclareFontShape{U}{mathx}{m}{n}{
      <5> <6> <7> <8> <9> <10>
      <10.95> <12> <14.4> <17.28> <20.74> <24.88>
      mathx10
      }{}
\DeclareSymbolFont{mathx}{U}{mathx}{m}{n}
\DeclareMathSymbol{\bigtimes}{1}{mathx}{"91}

\begin{document}
\title{An arithmetic enrichment of B\'ezout's Theorem}

\author{Stephen McKean}

\address{Department of Mathematics \\ Duke University \\ Durham \\ NC} 
\email{mckean@math.duke.edu}
\urladdr{services.math.duke.edu/~mckean}

\begin{abstract}
The classical version of B\'ezout's Theorem gives an integer-valued count of the intersection points of hypersurfaces in projective space over an algebraically closed field. Using work of Kass and Wickelgren, we prove a version of B\'ezout's Theorem over any perfect field by giving a bilinear form-valued count of the intersection points of hypersurfaces in projective space. Over non-algebraically closed fields, this enriched B\'ezout's Theorem imposes a relation on the gradients of the hypersurfaces at their intersection points. As corollaries, we obtain arithmetic-geometric versions of B\'ezout's Theorem over the reals, rationals, and finite fields of odd characteristic.
\end{abstract}

\maketitle

\section{Introduction}
In this paper, we study the intersections of $n$ hypersurfaces in projective $n$-space over an arbitrary perfect field $k$. Classically, B\'ezout's Theorem addresses such intersections over an algebraically closed field.

\begin{thm}[B\'ezout's Theorem]\label{thm:bezout}
Fix an algebraically closed field $k$. Let $f_1,\ldots,f_n$ be hypersurfaces in $\mb{P}^n$, and let $d_i$ be the degree of $f_i$ for each $i$. Assume that $f_1,\ldots,f_n$ have no common components, so that $f_1\cap\ldots\cap f_n$ is a finite set. Then, summing over the intersection points of $f_1,\ldots,f_n$, we have
\begin{align}\label{eq:bezout}
\sum_{\text{points}}i_p(f_1,\ldots,f_n)=d_1\cdots d_n,
\end{align}
where $i_p(f_1,\ldots,f_n)$ is the intersection multiplicity of $f_1,\ldots,f_n$ at $p$.
\end{thm}

Working over an algebraically closed field is necessary for this result.\footnote{Over non-algebraically closed fields, one may modify Equation~\ref{eq:bezout} by multiplying the intersection multiplicity $i_p(f_1,\ldots,f_n)$ by the degree of the residue field $[k(p):k]$ as described in~\cite[Proposition 8.4]{Ful98}. However, since each point $p$ splits into $[k(p):k]$ points in the algebraic closure of $k$, this simply counts the geometric intersection points as in Theorem~\ref{thm:bezout}.} Indeed, consider the intersection of a conic and a cubic shown in Figure~\ref{fig:conic-cubic}. Over any field, these two curves do not intersect on the line at infinity. Over $\mb{R}$, these two curves intersect exactly twice, with intersection multiplicity one at each of the intersection points. This number falls short of the six complex intersection points, even when counted with multiplicity. The results of this paper include a version of B\'ezout's Theorem over $\mb{R}$, which will impose a relation on the gradients of these curves at their intersection points.

\begin{figure}[h]
\centering
\begin{tikzpicture}
      \draw[->] (-2,0) -- (2,0);
      \draw[->] (0,-2) -- (0,2);
      \draw [cyan, ultra thick] (0,0) circle [radius=1.4142];
      \draw[domain=-1.2:1.2,smooth,variable=\x,red,ultra thick]  plot ({\x},{\x*\x*\x});
\end{tikzpicture}
\caption{A conic and a cubic over $\mb{R}$.}\label{fig:conic-cubic}
\end{figure}
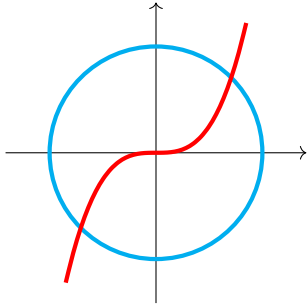

Our approach to generalize B\'ezout's Theorem follows the general philosophy of~\cite{KW17}. Any section $\sigma$ of the vector bundle $\mc{O}_{d_1,\ldots,d_n}:=\bigoplus_{i=1}^n\mc{O}(d_i)\to\mb{P}^n$ determines an $n$-tuple $(f_1,\ldots,f_n)$ of homogeneous polynomials of degree $d_1,\ldots,d_n$, respectively. The vanishing of each $f_i$ gives a hypersurface of $\mb{P}^n$, which we will also denote $f_i$. The section $\sigma$ vanishes precisely when $f_1,\ldots,f_n$ intersect, which suggests a connection to B\'ezout's Theorem.

$\mb{A}^1$-homotopy theory provides a powerful tool with which to study such sections. Morel developed an $\mb{A}^1$-homotopy theoretic analog of the local Brouwer degree~\cite{Mor12}, which Kass and Wickelgren used to study an Euler class $e$ of vector bundles in the context of enumerative algebraic geometry~\cite{KW17}.\footnote{There are various related Euler classes in arithmetic geometry, such as those appearing in \cite{AF16,BM00,Fas08,GI80,Lev17,Mor12}. See~\cite[Section 1.1]{KW17} for a discussion.} When $\mc{O}_{d_1,\ldots,d_n}$ is relatively orientable over $\mb{P}^n$ (that is, when $\sum_{i=1}^n d_i\equiv n+1\mod 2$), we compute $e(\mc{O}_{d_1,\ldots,d_n})$, which gives an equation involving the sum of local $\mb{A}^1$-degrees of a generic section at its points of vanishing. We also give a geometric description of the local $\mb{A}^1$-degree for transverse sections of $\mc{O}_{d_1,\ldots,d_n}$. This geometric information, paired with the equation coming from $e(\mc{O}_{d_1,\ldots,d_n})$, generalizes B\'ezout's Theorem. When $\mc{O}_{d_1,\ldots,d_n}$ is not relatively orientable over $\mb{P}^n$ (that is, when $\sum_{i=1}^n d_i\not\equiv n+1\mod 2$), we give a relative orientation of $\mc{O}_{d_1,\ldots,d_n}$ relative to the divisor $D=\{x_0=0\}$ in the sense of Larson and Vogt~\cite{LV19}. This allows us to compute the local degree of sections that do not vanish on $D$. However, we do not address the question of Euler classes in the non-relatively orientable case.

The local $\mb{A}^1$-degree is valued in the Grothendieck--Witt group $\GW(k)$ of symmetric, non-degenerate bilinear forms over $k$, so our enriched version of B\'ezout's Theorem will be an equality in $\GW(k)$. We assume throughout this paper that $k$ is a perfect field, which ensures that all algebraic extensions of $k$ are separable.

\begin{thm}\label{thm:main-orientable}
Let $\sum_{i=1}^n d_i\equiv n+1\mod 2$, and let $f_1,\ldots,f_n$ be hypersurfaces of $\mb{P}^n$ of degree $d_1,\ldots,d_n$ that intersect transversely. Given an intersection point $p$ of $f_1,\ldots,f_n$, let $J(p)$ be the signed volume of the parallelpiped determined by the gradient vectors of $f_1,\ldots,f_n$ at $p$. Then summing over the intersection points of $f_1,\ldots,f_n$, we have
\begin{align}\label{eq:main-orientable}
\sum_{\text{points}}\Tr_{k(p)/k}\langle J(p)\rangle=\frac{d_1\cdots d_n}{2}\cdot\mb{H},
\end{align}
where $\Tr_{k(p)/k}:\GW(k(p))\to\GW(k)$ is given by post-composing with the field trace.
\end{thm}

Taking the rank, signature, and discriminant of the Equation~\ref{eq:main-orientable} gives us B\'ezout's Theorem over $\mb{C},\mb{R}$, and $\mb{F}_q$, respectively. We apply similar techniques to also study B\'ezout's Theorem over $\mb{C}((t))$ and $\mb{Q}$. Kass and Wickelgren showed that Morel's local $\mb{A}^1$-degree is equivalent to a class of Eisenbud, Levine, and Khimshiashvili~\cite{KW19}, which allows us to make the necessary computations without explicitly using $\mb{A}^1$-homotopy theory. This paper fits into the growing field of $\mb{A}^1$-enumerative geometry, which is the  enrichment of classical theorems from enumerative geometry via $\mb{A}^1$-homotopy theory. Related results include \cite{KW17,Lev17,BKW18,SW18,Wen18,KW19,Lev19,LV19,BW20}.

The layout of the paper is as follows. In Section~\ref{sec:notation}, we introduce notation and conventions for the paper. In Section~\ref{sec:rel orientations}, we recall definitions and make computations about {\it relative orientability}, which is necessary for computations and proofs in Sections~\ref{sec:euler class} and \ref{sec:local degree}. In Section~\ref{sec:euler class}, we calculate the Euler class, and we discuss the geometric information carried by the local degree in Section~\ref{sec:local degree}. Finally, we discuss B\'ezout's Theorem over $\mb{C}$, $\mb{R}$, finite fields of odd characteristic, $\mb{C}((t))$, and $\mb{Q}$ in Section~\ref{sec:examples}. To illustrate the sort of obstructions that B\'ezout's Theorem over $\mb{Q}$ provides, we show in Example~\ref{ex:over Q} that if a line and a conic in $\mb{P}^2_\mb{Q}$ meet at two distinct points, and if the area of the parallelogram determined by the normal vectors to these curves at one of the intersection points is a non-square integer $m\neq-1$, then the area of the parallelogram at the other intersection point cannot be an integer prime to $m$.

\subsection{Related work} Chen~\cite[Section 3]{Che84} studies B\'ezout's Theorem in $\mb{P}^{2n}_{\mb{R}}$ as a consequence of a generalized B\'ezout's Theorem over $\mb{C}$~\cite[Theorem 2.1]{Che84}. In particular, Chen discusses that over $\mb{R}$, intersection multiplicities can be negative numbers~\cite[Remark 2.2]{Che84} and shows that if $f,g\in\mb{R}^2$ meet transversely at $p$, then the $\mb{R}$-intersection multiplicity of $f$ and $g$ at $p$ is the sign of the Jacobian of $f$ and $g$~\cite[Proposition 3.1]{Che84}. Our work in Section~\ref{sec:over R} generalizes these latter observations.

\subsection{Acknowledgements}
I am extremely grateful to Kirsten Wickelgren for introducing me to motivic homotopy theory, as well as for her excellent guidance and feedback on this project. I am also grateful to Thomas Brazelton, Srikanth Iyengar, Jesse Kass, Sabrina Pauli, Libby Taylor, and Isabel Vogt for helpful discussions. I thank the anonymous referee for their careful reading and constructive comments, which have improved the clarity of this paper.

\subsection{Conflicts of interest}
The author states that there are no conflicts of interest.

\subsection{Data availability}
Data sharing not applicable to this article as no data sets were generated or analyzed.

\section{Notation and conventions}\label{sec:notation}
Throughout this paper, we let $k$ be a perfect field. We denote projective $n$-space over $k$ by $\mb{P}^n_k=\Proj(k[x_0,\ldots,x_n])$. When the base field is clear from context, we may write $\mb{P}^n$ instead of $\mb{P}^n_k$. Given a rank $r$ vector bundle $E$, the {\it determinant bundle} of $E$ is the $r$-fold wedge product
\[\det{E}=\underbrace{E\wedge\cdots\wedge E}_{r\text{ times}}.\]

\subsection{Standard cover}\label{sec:standard cover}
Let $U_0,\ldots,U_n$ be the standard affine open subspaces of $\mb{P}^n$ given by $U_i=\{\pcoor{p_0:\cdots:p_n}\in\mb{P}^n:p_i\neq 0\}$. Let $\vphi_0,\ldots,\vphi_n$ be the standard local coordinates of $U_0,\ldots,U_n$, where $\vphi_i:U_i\to\mb{A}^n$ is given by $\vphi_i(\pcoor{p_0:\cdots:p_n})=(\frac{p_0}{p_i},\ldots,\frac{p_{i-1}}{p_i},\frac{p_{i+1}}{p_i},\ldots,\frac{p_n}{p_i})$. We call $\{(U_i,\vphi_i)\}$ the {\it standard cover} of $\mb{P}^n$.

\subsection{Twisting sheaves}\label{sec:twisting sheaves}
We denote the twisting sheaf $\mc{O}_{\mb{P}^n}(d\{x_0=0\})$ by $\mc{O}(d)$. Under this definition, we remark that $\mc{O}(d)$ is locally trivialized by $(\frac{x_i}{x_0})^d$ over $U_i$. If $d\geq 0$, the vector space of global sections $H^0(\mb{P}^n,\mc{O}(d))$ is isomorphic to the vector space of homogeneous polynomials in $k[x_0,\ldots,x_n]$ of degree $d$. Indeed, given $h\in k[x_0,\ldots,x_n]_{(d)}$, we have a global section $\sigma$ of $\mc{O}(d)$, which is given in the local trivializations by $\sigma|_{U_i}=h/x_i^d$.

In this paper, we will often consider global sections of $\mc{O}_{d_1,\ldots,d_n}:=\bigoplus_{i=1}^n\mc{O}(d_i)$. By the above identification of $H^0(\mb{P}^n,\mc{O}(d))$ and $k[x_0,\ldots,x_n]_{(d)}$, we may thus write a section as $\sigma=(f_1,\ldots,f_n)$, where $f_i\in k[x_0,\ldots,x_n]_{(d_i)}$.

\subsection{Grothendieck--Witt groups}\label{sec:gw}
The {\it Grothendieck--Witt group} $\GW(k)$ is the group completion of the monoid of isomorphism classes of symmetric, non-degenerate bilinear forms over $k$, where the group law is given by direct sum. The Grothendieck-Witt group is in fact a ring, where multiplication comes from multiplication of bilinear forms. See e.g.~\cite{Lam05} for the case where $\Char{k}\neq 2$.

Given $a\in k^\times$, we denote by $\langle a\rangle$ the isomorphism class of the bilinear form $(x,y)\mapsto axy$. It is a fact that $\GW(k)$ is generated by all such $\langle a\rangle$, subject to the following relations~\cite[Chapter II, Theorem 4.1]{Lam05}:
\begin{enumerate}[(i)]
\item $\langle ab^2\rangle=\langle a\rangle$ for all $a,b\in k^\times$.
\item $\langle a\rangle\langle b\rangle=\langle ab\rangle$ for all $a,b\in k^\times$.
\item $\langle a\rangle+\langle b\rangle=\langle a+b\rangle+\langle ab(a+b)\rangle$ for all $a,b\in k^\times$ such that $a+b\neq 0$.
\item $\langle a\rangle+\langle -a\rangle=\langle 1\rangle+\langle -1\rangle$ for all $a\in k^\times$.
\end{enumerate}

Relation (iv) is actually redundant, but it is useful to know. We will use the notation $\mb{H}:=\langle 1\rangle+\langle -1\rangle$, as this bilinear form will appear frequently. By relations (ii) and (iv), we note that $\langle a\rangle\cdot\mb{H}=\mb{H}$ and $\mb{H}\cdot\mb{H}=2\cdot\mb{H}$. Both the ring multiplication of $\GW(k)$ and the integer multiplication of $\GW(k)$ as an abelian group under addition may be denoted by $\cdot$ or by juxtaposition of symbols, whichever is presently more visually appealing or less confusing.

\section{Relative orientations}\label{sec:rel orientations}
Let $f_1,\ldots,f_n$ be hypersurfaces in $\mb{P}^n$. B\'ezout's Theorem equates a {\it fixed value} with the sum (over the intersection locus of $f_1,\ldots,f_n$) of some {\it geometric information} about $f_1,\ldots,f_n$ at each intersection point. Classically (that is, over an algebraically closed field), the {\it fixed value} is the product of the degrees of each $f_i$, and the {\it geometric information} at each intersection point is the intersection multiplicity of $f_1,\ldots,f_n$. Over an arbitrary perfect field, an $\mb{A}^1$-homotopy theoretic Euler class will give us a particular bilinear form as our {\it fixed value}, and the local $\mb{A}^1$-degree will give us our {\it geometric information}. We compute the Euler class in Section~\ref{sec:euler class}, and we discuss the local degree in Section~\ref{sec:local degree}. In this section, we recall definitions and make computations that are required for Sections~\ref{sec:euler class} and \ref{sec:local degree}. We first start with some definitions.

\begin{defn}\label{defn:rel orientation}\cite[Definition 16]{KW17}
A {\it relative orientation} of a vector bundle $V$ on a scheme $X$ is a pair $(L,j)$ of a line bundle $L$ and an isomorphism $j:L^{\otimes 2}\to\Hom(\det\mc{T}X,\det{V})$, where $\mc{T}X\to X$ is the tangent bundle. We say that $V$ is {\it relatively orientable} if $V$ has a relative orientation. Moreover, on an open set $U\subseteq X$, a section of $\Hom(\det\mc{T}X,\det{V})$ is called a {\it square} if its image under $H^0(U,\Hom(\det\mc{T}X,\det{V}))\cong H^0(U,L^{\otimes 2})$ is a tensor square of an element in $H^0(U,L)$.
\end{defn}

The relative orientability of the vector bundle $\mc{O}_{d_1,\ldots,d_n}\to\mb{P}^n$ depends on $d_1,\ldots,d_n,$ and $n$ in the following way.

\begin{prop}\label{prop:rel orientable}
The vector bundle $\mc{O}_{d_1,\ldots,d_n}\to\mb{P}^n$ is relatively orientable if and only if $\sum_{i=1}^n d_i\equiv n+1\mod 2$.
\end{prop}
\begin{proof}
Since $\det\mc{O}_{d_1,\ldots,d_n}\cong\bigotimes_{i=1}^n\mc{O}(d_i)$, we have that
\begin{align*}
\Hom(\det\mc{T}\mb{P}^n,\det\mc{O}_{d_1,\ldots,d_n})&\cong\det\mc{O}_{d_1,\ldots,d_n}\otimes(\det\mc{T}\mb{P}^n)^\vee\\
&\cong\mc{O}(-n-1+{\textstyle\sum}_{i=1}^n d_i).
\end{align*}
Thus $\mc{O}(-n-1+\sum_{i=1}^n d_i)$ is a square if and only if $-n-1+\sum_{i=1}^n d_i$ is even, in which case $\mc{O}(-n-1+\sum_{i=1}^n d_i)\cong\mc{O}((-n-1+\sum_{i=1}^n d_i)/2)^{\otimes 2}$.
\end{proof}

\begin{rem}\label{rem:rel orient- not all odd}
We note that if $\sum_{i=1}^n d_i\equiv n+1\mod 2$, then at least one of $d_1,\ldots,d_n$ must be even. Indeed, suppose all of $d_1,\ldots,d_n$ are odd. Then $d_i\equiv 1\mod 2$, so $\sum_{i=1}^n d_i\equiv n\mod 2$.
\end{rem}

When $V$ is not relatively orientable, we have the following definition of Larson and Vogt.

\begin{defn}\cite[Definition 2.2]{LV19}
A {\it relative orientation relative to an effective divisor $D$} of a vector bundle $V$ on a smooth projective scheme $X$ is a pair $(L,j)$ of a line bundle $L$ and an isomorphism $j:L^{\otimes 2}\to\Hom(\det\mc{T}X,\det{V})\otimes\mc{O}(D)$.
\end{defn}

In $\mb{A}^1$-homotopy theory, one frequently uses the {\it Nisnevich topology}. For this paper, we will only need the following definitions.

\begin{defn}\cite[Definition 17]{KW17}
Let $X$ be a scheme of dimension $n$, and let $U\subseteq X$ be an open neighborhood of a point $p\in X$. An \'etale map $\vphi:U\to\mb{A}^n_k$ is called {\it Nisnevich coordinates} about $p$ if $\vphi$ induces an isomorphism between the residue field of $p$ and the residue field of $\vphi(p)$.
\end{defn}

\begin{defn}\cite[Definition 19]{KW17}
Let $V$ be a vector bundle on a scheme $X$, and let $U\subseteq X$ be an open affine subset. Given Nisnevich coordinates $\vphi$ on $U$ and a relative orientation $(L,j)$ of $V$, we have a distinguished basis element of $\det\mc{T}X|_U$. A local trivialization of $V|_U$ is called {\it compatible} with the Nisnevich coordinates and relative orientation if the element of $\Hom(\det\mc{T}X|_U,\det{V}|_U)$ taking the distinguished basis element of $\det\mc{T}X|_U$ to the distinguished basis element of $\det{V}|_U$ (determined by the specified local trivialization of $V|_U$) is a square (in the sense of Definition~\ref{defn:rel orientation}).
\end{defn}

We can generalize the above definition to discuss {\it compatibility} in the case of a relative orientation relative to an effective Cartier divisor.

\begin{defn}
Let $V$ be a vector bundle on a smooth projective scheme $X$, and let $U\subseteq X$ be an open affine subset. Given Nisnevich coordinates $\vphi$ on $U$ and a relative orientation $(L,j)$ relative to an effective Cartier divisor $D$, we have a distinguished basis element of $\det\mc{T}X|_U$. A local trivialization of $V|_U$ is called {\it compatible} with the Nisnevich coordinates and relative orientation relative to $D$ if $\alpha\otimes 1_D$ is a square, where $1_D$ is the canonical section of $\mc{O}(D)$~\cite[\href{https://stacks.math.columbia.edu/tag/01WX}{Definition 01WX (2)}]{stacks} and $\alpha$ is the element of $\Hom(\det\mc{T}X|_U,\det{V}|_U)$ taking the distinguished basis element of $\det\mc{T}X|_U$ to the distinguished basis element of $\det{V}|_U$ (determined by the specified local trivialization of $V|_U$).
\end{defn}

We will show that a twist of the standard cover $\{(U_i,\vphi_i)\}$ of $\mb{P}^n$ (see Section~\ref{sec:standard cover}) gives Nisnevich coordinates. This twist will be denoted $\{(U_i,\tilde{\vphi}_i)\}$, with $\tilde{\vphi}_0=\vphi_0$ and
\[\tilde{\vphi}_i(\pcoor{p_0:\cdots:p_n})=((-1)^i\tfrac{p_0}{p_i},\ldots,\tfrac{p_{i-1}}{p_i},\tfrac{p_{i+1}}{p_i},\ldots,\tfrac{p_n}{p_i}).\]
The reason for working with these twisted coordinates instead of the standard coordinates is to ensure compatibility with the local trivializations of $\mc{O}_{d_1,\ldots,d_n}$, as shown in Lemmas~\ref{lem:compatible-orientable case} and~\ref{lem:compatible-non-orientable}. We will also describe the distinguished basis elements of $\det\mc{T}\mb{P}^n|_{U_i}$ and $\det\mc{O}_{d_1,\ldots,d_n}|_{U_i}$ coming from $\tilde{\vphi}_i$ and the local trivialization given in Section~\ref{sec:twisting sheaves}, respectively.

\begin{prop}\label{prop:nisnevich coords}
The twisted covering maps $\tilde{\vphi}_i:U_i\to\mb{A}^n_k$ are Nisnevich coordinates. Moreover, $\vphi_i$ determines the distinguished basis element $(-1)^i\cdot\partial_i:=(-1)^i\bigwedge_{j\neq i}\frac{\partial}{\partial(x_j/x_i)}$ of $\det\mc{T}\mb{P}^n|_{U_i}$ with transition functions $\det{g_{ij}}:=(-1)^{i+j}(\tfrac{x_i}{x_j})^{n+1}$.
\end{prop}
\begin{proof}
By construction, $\tilde{\vphi}_i:U_i\to\mb{A}^n_k$ is an isomorphism, so $\tilde{\vphi}_i$ is \'etale and induces an isomorphism $k(p)\cong k(\vphi(p))$ for all $p\in U_i$. Recall that $\mc{T}\mb{A}^n$ has the standard trivializations $\{\frac{\partial}{\partial x_1},\ldots,\frac{\partial}{\partial x_n}\}$. Since $\tilde{\vphi}_i$ induces an isomorphism $\mc{T}\mb{P}^n|_{U_i}\cong\mc{T}\mb{A}^n$, we may pull back the standard trivializations $\mc{T}\mb{A}^n\to\mb{A}^n$ by $\tilde{\vphi}_i$ to obtain the twisted trivializations $\{(-1)^i\partial_{0/i},\ldots,\partial_{(i-1)/i},\partial_{(i+1)/i},\ldots,\partial_{n/i}\}$, where $\partial_{j/i}=\frac{\partial}{\partial(x_j/x_i)}$. It follows that $\det\mc{T}\mb{P}^n|_{U_i}$ is trivialized by $(-1)^i\bigwedge_{j\neq i}\partial_{j/i}$. Finally, we  consider the transition functions $\det g_{ij}:\det\mc{T}\mb{P}^n|_{U_j}\to\det\mc{T}\mb{P}^n|_{U_i}$. These transition functions will come from the transition functions $g_{ij}:\mc{T}\mb{P}^n|_{U_j}\to\mc{T}\mb{P}^n|_{U_i}$. A few calculus computations show us that, for $k\neq i,j$, we have
\begin{align*}
\partial_{k/i}&=\tfrac{x_i}{x_j}\cdot\partial_{k/j},\\
\partial_{j/i}&=-(\tfrac{x_i}{x_j})^2\cdot\partial_{i/j}-\sum_{k\neq i,j}\tfrac{x_ix_k}{x_j^2}\cdot\partial_{k/j}.
\end{align*}

Thus for fixed $i,j$, we have $\bigwedge_{k\neq i}\partial_{k/i}=-(\frac{x_i}{x_j})^{n+1}\bigwedge_{k\neq j}\partial_{k/j}$. The trivializations $(-1)^i\cdot\partial_i:=(-1)^i\bigwedge_{j\neq i}\partial_{j/i}$ of $\det\mc{T}\mb{P}^n|_{U_i}$ are local trivializations of $\det\mc{T}\mb{P}^n$ compatible with the transition functions $\det g_{ij}=(-1)^{i+j}(\frac{x_i}{x_j})^{n+1}$. In other words, $\det\mc{T}\mb{P}^n|_{U_i}$ is one-dimensional with $(-1)^i\cdot\partial_i$ as its distinguished basis element.
\end{proof}

\begin{prop}\label{prop:local trivs}
The local trivialization $(\frac{x_i}{x_0})^{d_1}\oplus\cdots\oplus(\frac{x_i}{x_0})^{d_n}$ of $\mc{O}_{d_1,\ldots,d_n}|_{U_i}$ determines the distinguished basis element $(\frac{x_i}{x_0})^{d_1+\ldots+d_n}$ of $\det\mc{O}_{d_1,\ldots,d_n}|_{U_i}$ with transition functions $\det{h_{ij}}:=(\tfrac{x_i}{x_j})^{d_1+\ldots+d_n}$.
\end{prop}
\begin{proof}
Since $\mc{O}(d)|_{U_i}$ is trivialized by $(\frac{x_i}{x_0})^d$, the vector bundle $\mc{O}_{d_1,\ldots,d_n}|_{U_i}$ is trivialized by $(\frac{x_i}{x_0})^{d_1}\oplus\cdots\oplus(\frac{x_i}{x_0})^{d_n}$. The transition functions $h_{ij}:\mc{O}(d)|_{U_j}\to\mc{O}(d)|_{U_i}$ are given by $(\frac{x_i}{x_j})^d$, so the transition functions $\oplus h_{ij}:\mc{O}_{d_1,\ldots,d_n}|_{U_j}\to\mc{O}_{d_1,\ldots,d_n}|_{U_i}$ are given by $(\frac{x_i}{x_j})^{d_1}\oplus\cdots\oplus(\frac{x_i}{x_j})^{d_n}$. Finally, recall that $\det\mc{O}_{d_1,\ldots,d_n}\cong\mc{O}(d_1)\otimes\cdots\otimes\mc{O}(d_n)\cong\mc{O}(d_1+\ldots+d_n)$. Thus $\det\mc{O}_{d_1,\ldots,d_n}|_{U_i}$ is trivialized by $(\frac{x_i}{x_0})^{d_1}\otimes\cdots\otimes(\frac{x_i}{x_0})^{d_n}\cong(\frac{x_i}{x_0})^{d_1+\ldots+d_n}$, and the transition functions $\det h_{ij}:\det\mc{O}_{d_1,\ldots,d_n}|_{U_j}\to\det\mc{O}_{d_1,\ldots,d_n}|_{U_i}$ are given by $(\frac{x_i}{x_j})^{d_1}\otimes\cdots\otimes(\frac{x_i}{x_j})^{d_n}\cong(\frac{x_i}{x_j})^{d_1+\ldots+d_n}$.
\end{proof}

\subsection{Relatively orientable case}\label{sec:orientation-orientable case}
Let $N=-n-1+\sum_{i=1}^n d_i$, and assume $N\equiv 0\mod 2$, so that $\mc{O}_{d_1,\ldots,d_n}\to\mb{P}^n$ is relatively orientable by Proposition~\ref{prop:rel orientable}. We will give a relative orientation of $\mc{O}_{d_1,\ldots,d_n}$ and show that the local trivializations of $\mc{O}_{d_1,\ldots,d_n}$ discussed in Proposition~\ref{prop:local trivs} are compatible with this relative orientation and the Nisnevich coordinates coming from our twisted cover $\{(U_i,\tilde{\vphi}_i)\}$.

For our relative orientation, we give an isomorphism
\[\psi:\mc{O}(N/2)^{\otimes 2}\to\Hom(\det\mc{T}\mb{P}^n,\det\mc{O}_{d_1,\ldots,d_n})\]
by defining $\psi|_{U_i}$ for each $i$. Since $\mc{O}(N/2)^{\otimes 2}|_{U_i}$ is generated by $(\frac{x_i}{x_0})^{N/2}\otimes(\frac{x_i}{x_0})^{N/2}$, it suffices to define $\alpha_i:=\psi|_{U_i}((\frac{x_i}{x_0})^{N/2}\otimes(\frac{x_i}{x_0})^{N/2})$, which is a homomorphism from $\det\mc{T}\mb{P}^n|_{U_i}$ to $\det\mc{O}_{d_1,\ldots,d_n}|_{U_i}$. These are both one-dimensional as shown in Propositions~\ref{prop:nisnevich coords} and \ref{prop:local trivs}, so we may define $\alpha_i$ to be the homomorphism taking $(-1)^i\cdot\partial_i$ to $(\frac{x_i}{x_0})^{d_1+\ldots+d_n}$. To show that $\psi$ is well-defined, we need to show that on $U_i\cap U_j$, the maps $\psi|_{U_i}$ and $\psi|_{U_j}$ differ by the transition function $(\frac{x_i}{x_j})^{N/2}\otimes(\frac{x_i}{x_j})^{N/2}:\mc{O}(N/2)^{\otimes 2}|_{U_j}\to\mc{O}(N/2)^{\otimes 2}|_{U_i}$. In other words, we need to show that $\alpha_i=(\frac{x_i}{x_j})^N\alpha_j$ on $U_i\cap U_j$. To this end, let $\det g_{ij}$ and $\det h_{ij}$ be the transition functions given in Propositions~\ref{prop:nisnevich coords} and \ref{prop:local trivs} and note that
\begin{align*}
\alpha_i\circ\det g_{ij}((-1)^j\cdot\partial_j)&=\alpha_i((-1)^i(\tfrac{x_j}{x_i})^{n+1}\cdot\partial_i)\\
&=(\tfrac{x_j}{x_i})^{n+1}(\tfrac{x_i}{x_0})^{d_1+\ldots+d_n}\\
&=(\tfrac{x_j}{x_i})^{n+1}\det h_{ij}(\tfrac{x_j}{x_0})^{d_1+\ldots+d_n}\\
&=(\tfrac{x_i}{x_j})^{-n-1}(\tfrac{x_i}{x_j})^{d_1+\ldots+d_n}(\tfrac{x_j}{x_0})^{d_1+\ldots+d_n}\\
&=(\tfrac{x_i}{x_j})^{N}\alpha_j((-1)^j\cdot\partial_j).
\end{align*}

Thus $\alpha_i=(\frac{x_i}{x_j})^N\alpha_j$, as desired. In fact, we have proved the following lemma.

\begin{lem}\label{lem:compatible-orientable case}
The local trivializations $(\frac{x_i}{x_0})^{d_1}\oplus\cdots\oplus(\frac{x_i}{x_0})^{d_n}$ of $\mc{O}_{d_1,\ldots,d_n}|_{U_i}$ are compatible with the Nisnevich coordinates $\{(U_i,\tilde{\vphi}_i)\}$ and the relative orientation $(\mc{O}(N/2),\psi)$ of $\mc{O}_{d_1,\ldots,d_n}\to\mb{P}^n$.
\end{lem}
\begin{proof}
By construction, $\alpha_i$ is the element of $\Hom(\det\mc{T}\mb{P}^n|_{U_i},\det\mc{O}_{d_1,\ldots,d_n}|_{U_i})$ taking the distinguished basis element $(-1)^i\cdot\partial_i$ of $\det\mc{T}\mb{P}^n|_{U_i}$ to the distinguished basis element $(\frac{x_i}{x_0})^{d_1+\ldots+d_n}$ of $\det\mc{O}_{d_1,\ldots,d_n}|_{U_i}$. The relative orientation $(\mc{O}(N/2),\psi)$ was built such that $\psi|_{U_i}((\frac{x_i}{x_0})^{N/2}\otimes(\frac{x_i}{x_0})^{N/2})=\alpha_i$, so $\alpha_i$ is a tensor square in $\mc{O}(N/2)^{\otimes 2}|_{U_i}$.
\end{proof}

\subsection{Non-relatively orientable case}\label{sec:orientation-non-orientable case}
Let $N=-n-1+\sum_{i=1}^n d_i$, and assume $N\not\equiv 0\mod 2$. In this case, $\mc{O}_{d_1,\ldots,d_n}\to\mb{P}^n$ is not relatively orientable, since there is no line bundle of the form $\mc{O}(N/2)$ when $N/2$ is not an integer. However, we will show that $\mc{O}_{d_1,\ldots,d_n}\to\mb{P}^n$ is relatively orientable relative to the effective Cartier divisor $D=\{x_0=0\}$ of $\mb{P}^n$. Figuratively, this divisor gives us a geometric horizon relative to which we can orient our hypersurfaces in projective space.

We have chosen the divisor $D=\{x_0=0\}$ so that the local trivializations of $\mc{O}(D)$ work nicely with our other twisting sheaves. In particular, we have $\Hom(\det\mc{T}\mb{P}^n,\det\mc{O}_{d_1,\ldots,d_n})\otimes\mc{O}(D)\cong\mc{O}(N+1)$. Since $N+1\equiv 0\mod 2$, the bundle $\Hom(\det\mc{T}\mb{P}^n,\det\mc{O}_{d_1,\ldots,d_n})\otimes\mc{O}(D)$ is the tensor square of the line bundle $\mc{O}(\frac{N+1}{2})$. We may thus apply the work of Section~\ref{sec:orientation-orientable case} to get a relative orientation $(\mc{O}(\frac{N+1}{2}),\tilde{\psi})$ of $\mc{O}_{d_1,\ldots,d_n}$ relative to the divisor $D$, as well as local trivializations of $\mc{O}_{d_1,\ldots,d_n}$ compatible with our Nisnevich coordinates $\{(U_i,\tilde{\vphi}_i)\}$ and our relative orientation $(\mc{O}(\frac{N+1}{2}),\tilde{\psi})$.

\begin{lem}\label{lem:compatible-non-orientable}
The local trivialization $(\frac{x_i}{x_0})^{d_1}\oplus\cdots\oplus(\frac{x_i}{x_0})^{d_n}$ of $\mc{O}_{d_1,\ldots,d_n}|_{U_i}$ are compatible with the Nisnevich coordinates $\{(U_i,\tilde{\vphi}_i)\}$ and the relative orientation $(\mc{O}(\frac{N+1}{2}),\tilde{\psi})$, where $\tilde{\psi}$ is given locally by $\tilde{\psi}|_{U_i}((\frac{x_i}{x_0})^{(N+1)/2}\otimes(\frac{x_i}{x_0})^{(N+1)/2})=\alpha_i\otimes\frac{x_i}{x_0}$.
\end{lem}
\begin{proof}
The canonical section $1_D$ of $\mc{O}(D)$ is locally given by $\frac{x_i}{x_0}$. By construction, $\alpha_i((-1)^i\cdot\partial_i)\otimes\frac{x_i}{x_0}=(\frac{x_i}{x_0})^{d_1+\ldots+d_n+1}$, so we have $\alpha_i\otimes\frac{x_i}{x_0}=(\frac{x_i}{x_j})^{N+1}\alpha_j\otimes\frac{x_j}{x_0}$ on $U_i\cap U_j$ by Lemma~\ref{lem:compatible-orientable case}. Thus the maps $\tilde{\psi}|_{U_i}$ and $\tilde{\psi}|_{U_j}$ differ by the transition function $(\frac{x_i}{x_j})^{(N+1)/2}\otimes(\frac{x_i}{x_j})^{(N+1)/2}:\mc{O}(\frac{N+1}{2})^{\otimes 2}|_{U_j}\to\mc{O}(\frac{N+1}{2})^{\otimes 2}|_{U_i}$, so the relative orientation $(\mc{O}(\frac{N+1}{2}),\tilde{\psi})$ relative to the divisor $D$ is well-defined. This relative orientation was constructed such that $\alpha_i\otimes\frac{x_i}{x_0}$ is a square.
\end{proof}

\section{Euler class}\label{sec:euler class}
In Section~\ref{sec:rel orientations}, we gave a relative orientation (possibly relative to an effective Cartier divisor) of $\mc{O}_{d_1,\ldots,d_n}\to\mb{P}^n$, as well as local trivializations of this bundle compatible with the given relative orientation and our twisted Nisnevich coordinates. This data allows us to compute the local degree of sections of $\mc{O}_{d_1,\ldots,d_n}$. We can also define an Euler class of this vector bundle, paired with a given section, by computing the sum of local degrees of the section. When this sum does not depend on our choice of section, the Euler class gives us an invariant associated to the vector bundle at hand. This invariant will correspond to the enumerative {\it fixed value} discussed at the beginning of Section~\ref{sec:rel orientations}.

We first discuss how to compute the local degree of a section. Let $V$ be a vector bundle on a scheme $X$ of dimension $n$, and suppose that we have Nisnevich coordinates, a relative orientation (possibly relative to an effective Cartier divisor) of $V$, and compatible local trivializations of $V$. If $\sigma$ is a section of $V$ with isolated zero $p\in X$, then take an open affine $U\subseteq X$ containing $p$. Under the compatible local trivialization of $V|_U$, the section $\sigma|_U$ becomes an $n$-tuple of functions $(f_1,\ldots,f_n):U\to\mb{A}^n_k$. If $\vphi:U\to\mb{A}^n_k$ are the aforementioned Nisnevich coordinates, and if $\vphi|_U$ is an isomorphism, then $(f_1,\ldots,f_n)\circ\vphi^{-1}$ is an endomorphism of $\mb{A}^n_k$. (In general, $\vphi|_U$ may not be an isomorphism, in which case we cannot write the section $\sigma$ as a polynomial map. We can, however, write our section as a polynomial map with a negligible error term. See~\cite[Lemmas 24--28]{KW17} for details.)

We may thus compute the local degree of this endomorphism as outlined in~\cite[Table 1]{KW19}. Kass and Wickelgren~\cite[Corollary 29]{KW17} also show that given a relative orientation of $V$, the Euler class $e$ does not depend on the choice of Nisnevich coordinates on $X$ with compatible trivialization of $V$. This allows us to define $\deg_p\sigma=\deg_{\vphi(p)}(f_1,\ldots,f_n)\circ\vphi^{-1}$. Note that we must choose our neighborhood $U$ sufficiently small, so that $\vphi^{-1}(\vphi(p))=\{p\}$.

\begin{defn}\cite[Definition 33]{KW17}
Given a relatively oriented (relative to an effective Cartier divisor $D$) vector bundle $V\to X$ and a section $\sigma$ with isolated zero locus (such that $\sigma$ does not vanish on $D$), define the {\it Euler number} of $(V,\sigma)$ to be
\[e(V,\sigma)=\sum_{p\in\sigma^{-1}(0)}\deg_p\sigma.\]
\end{defn}

When $e(V,\sigma)$ does not depend on our choice of section $\sigma$, we will simply denote this by $e(V)$. It will turn out that $e(\mc{O}_{d_1,\ldots,d_n})$ does not depend on our choice of section when $\mc{O}_{d_1,\ldots,d_n}\to\mb{P}^n$ is relatively orientable. Let $N=-n-1+\sum_{i=1}^n d_i$, and assume $N\equiv 0\mod 2$ so that $\mc{O}_{d_1,\ldots,d_n}\to\mb{P}^n$ is relatively orientable. We will show that $e(\mc{O}_{d_1,\ldots,d_n},\sigma)$ does not depend on our choice of section. First, we need the following proposition.

\begin{prop}\label{prop:minimal prime}
Let $\Char{k}=0$. Let $n\geq 1$. If $(f_1,\ldots,f_{i-1})$ is a regular sequence in $k[x_0,\ldots,x_n]$ for $1\leq i\leq n$, then
\[Z_i:=\{f_i\in H^0(\mb{P}^n_k,\mc{O}(d_i)):(f_1,\ldots,f_{i-1},f_i)\text{ is not a regular sequence}\}\]
has $k$-codimension at least 2 in $H^0(\mb{P}^n_k,\mc{O}(d_i))$.
\end{prop}
\begin{proof}
By definition of regular sequences, $f_i$ is not a zero divisor in $\frac{k[x_0,\ldots,x_n]}{(f_1,\ldots,f_{i-1})}$ if and only if $(f_1,\ldots,f_{i-1},f_i)$ is a regular sequence in $k[x_0,\ldots,x_n]$. The set of zero divisors in $\frac{k[x_0,\ldots,x_n]}{(f_1,\ldots,f_{i-1})}$ is given by the union of the minimal prime ideals associated to the ideal $I:=(f_1,\ldots,f_{i-1})$. Moreover, since $k[x_0,\ldots,x_n]$ is Noetherian, there are finitely many minimal prime ideals associated to $I$. Given a minimal prime $\mf{p}$ associated to $I$, let $\mf{p}_{d_i}$ denote the degree $d_i$ part of $\mf{p}$, considered as a $k$-vector space. If
\[\codim_k\mf{p}_{d_i}:=\dim_k H^0(\mb{P}^n_k,\mc{O}(d_i))-\dim_k\mf{p}_{d_i}\geq 2\]
for any minimal prime $\mf{p}$ associated to $I$, then $Z_i$ is a finite union of spaces of codimension at least 2. It will then follow that $Z_i$ has codimension at least 2 in $H^0(\mb{P}^n_k,\mc{O}(d_i))$.

Let $\mf{p}$ be a minimal prime ideal associated to $I$. Krull's height theorem implies that $\mf{p}$ has height at most $i-1$, so $\mf{p}$ contains at most $i-1$ linear forms that are linearly independent over $k$. Suppose that $\codim_k\mf{p}_{d_i}<2$. Then $\dim_k\mf{p}_{d_i}\geq\binom{n+d_i}{d_i}-1$. If $\dim_k\mf{p}_{d_i}=\binom{n+d_i}{d_i}$, then $\mf{p}_{d_i}=H^0(\mb{P}^n_k,\mc{O}(d_i))$ and hence $x_0^{d_i},\ldots,x_n^{d_i}\in\mf{p}_{d_i}$. Since $\mf{p}$ is a prime ideal, it follows that $x_0,\ldots,x_n\in\mf{p}$, so $\mf{p}$ contains $n+1>i-1$ linear forms that are linearly independent. 

We may thus assume that $\dim_k\mf{p}_{d_i}=\binom{n+d_i}{d_i}-1$. Let $N=\binom{n+d_i}{d_i}$, and consider the Veronese embedding $v_{d_i}:\mb{P}^n_k\to\mb{P}^{N-1}_k$, where $\mb{P}^n_k\cong\mb{P}(H^0(\mb{P}^n_k,\mc{O}(1)))$ and $\mb{P}^{N-1}_k\cong\mb{P}(H^0(\mb{P}^n_k,\mc{O}(d_i)))$. Under the Veronese embedding, the image of $\ell\in\mb{P}(H^0(\mb{P}^n_k,\mc{O}(1)))$ is $v_{d_i}(\ell)=\ell^{d_i}\in\mb{P}(H^0(\mb{P}^n_k,\mc{O}(d_i)))$. Since $\mf{p}_{d_i}$ is a codimension 1 subspace of $H^0(\mb{P}^n_k,\mc{O}(d_i))$ by assumption, we have an isomorphism $\mb{P}(\mf{p}_{d_i})\cong H$ for some hyperplane $H\subset\mb{P}^{N-1}_k$. The image $v_{d_i}(\mb{P}^n_k)$ of the Veronese embedding is not contained in any hyperplane, so the hyperplane section $v_{d_i}(\mb{P}^n_k)\cap H$ has dimension $\dim{v_{d_i}(\mb{P}^n_k)}-1$. Since the Veronese embedding is an isomorphism onto its image, it follows that $v_{d_i}(\mb{P}^n_k)\cap H=v_{d_i}(X)$ for some $X\subset\mb{P}^n_k$ of dimension $n-1$. This allows us to pick general points $p_1,\ldots,p_n\in X$ such that $\{p_1,\ldots,p_n\}$ is not contained in any $(n-2)$-plane. Thus if $\ell_j\in H^0(\mb{P}^n_k,\mc{O}(1))$ is any lift of $p_j$ under 
\[H^0(\mb{P}^n_k,\mc{O}(1))\to\mb{P}(H^0(\mb{P}^n_k,\mc{O}(1)))\cong\mb{P}^n_k,\]
then $\ell_1,\ldots,\ell_n$ are linearly independent over $k$. Moreover, since $v_{d_i}(p_j)\in H$, we have that $\ell_j^{d_i}\in\mf{p}_{d_i}$. Since $\mf{p}$ is a prime ideal, it follows that $\ell_1,\ldots,\ell_n\in\mf{p}$, so $\mf{p}$ contains $n>i-1$ linear forms that are linearly independent. By contradiction, we conclude that $\codim_k\mf{p}_{d_i}\geq 2$.
\end{proof}

We can now prove that $e(\mc{O}_{d_1,\ldots,d_n},\sigma)$ does not depend on our choice of section.

\begin{lem}\label{lem:euler class-independent}
The Euler number $e(\mc{O}_{d_1,\ldots,d_n},\sigma)$ is independent of the choice of section $\sigma$.
\end{lem}
\begin{proof}
This follows from~\cite[Theorem 1.1]{BW20}. However, we will also give a more direct proof of this lemma assuming $\Char{k}=0$. Given a section $\sigma\in H^0(\mb{P}^n_k,\mc{O}_{d_1,\ldots,d_n})$, let $Z(\sigma)=\{p\in\mb{P}^n:\sigma(p)=0\}$ be its zero locus. We will show that $\{\sigma:Z(\sigma)\text{ is not isolated}\}$ has $k$-codimension at least 2 in $H^0(\mb{P}^n_k,\mc{O}_{d_1,\ldots,d_n})$. This will show that
\[H^0(\mb{P}^n_k,\mc{O}_{d_1,\ldots,d_n})\backslash\{\sigma:Z(\sigma)\text{ is not isolated}\}\]
is {\it connected by sections} in the sense of~\cite[Definition 37]{KW17}. As a result, \cite[Corollary 38]{KW17} will imply that $e(\mc{O}_{d_1,\ldots,d_n},\sigma)$ is independent of $\sigma$.

The zero locus $Z(\sigma)$ is isolated if and only if $(f_1,\ldots,f_n)$ is a regular sequence, so $Z(\sigma)$ is not isolated if and only if $f_i$ is a zero divisor in $\frac{k[x_0,\ldots,x_n]}{(f_1,\ldots,f_{i-1})}$ for some $i$. Note that $H^0(\mb{P}^n_k,\mc{O}_{d_1,\ldots,d_n})=\bigoplus_{i=1}^n H^0(\mb{P}^n_k,\mc{O}(d_i))$. Let $Z_i\subseteq H^0(\mb{P}^n_k,\mc{O}_{d_1,\ldots,d_i})$ be the set of all sections $(f_1,\ldots,f_i)$ such that $(f_1,\ldots,f_{i-1})$ is a regular sequence and $f_i$ is a zero divisor in $\frac{k[x_0,\ldots,x_n]}{(f_1,\ldots,f_{i-1})}$. Then the set of all non-regular sequences $(f_1,\ldots,f_n)$ is given by
\[\bigcup_{i=1}^n (Z_i\oplus H^0(\mb{P}^n_k,\mc{O}_{d_{i+1},\ldots,d_n})).\]
Proposition~\ref{prop:minimal prime} implies that $Z_i$ has codimension at least 2 in $H^0(\mb{P}^n_k,\mc{O}_{d_1,\ldots,d_i})$. It follows that $Z_i\oplus H^0(\mb{P}^n_k,\mc{O}(d_{i+1}))$ has codimension at least 2 in $H^0(\mb{P}^n_k,\mc{O}_{d_1,\ldots,d_{i+1}})$. Iterating this process, we have that $Z_i\oplus H^0(\mb{P}^n_k,\mc{O}_{d_{i+1},\ldots,d_n})$ has codimension at least 2 in $H^0(\mb{P}^n_k,\mc{O}_{d_1,\ldots,d_n})$ for $i=1,\ldots,n$. The set of all non-regular sequences is thus a finite union of sets of codimension at least 2, so $\{\sigma:Z(\sigma)\text{ is not isolated}\}$ has codimension at least 2.
\end{proof}

We may now compute $e(\mc{O}_{d_1,\ldots,d_n})$.

\begin{thm}\label{thm:euler class-orientable}
We have $e(\mc{O}_{d_1,\ldots,d_n})=\frac{d_1\cdots d_n}{2}\cdot\mb{H}$.
\end{thm}
\begin{proof}
By Lemma~\ref{lem:euler class-independent}, $e(\mc{O}_{d_1,\ldots,d_n})$ is independent of choice of section. Let $\sigma=(x_1^{d_1},\ldots,x_n^{d_n})$. The zero locus of $\sigma$ consists only of the point $p=\pcoor{1:0:\cdots:0}$, so $e(\mc{O}_{d_1,\ldots,d_n})=\deg_p\sigma$. Since $p\in U_0$, our twisted cover and local trivialization of $\mc{O}_{d_1,\ldots,d_n}$ on $U_0$ tell us that $\deg_p\sigma=\deg_{(0,\ldots,0)}((\frac{x_1}{x_0})^{d_1},\ldots,(\frac{x_n}{x_0})^{d_n})$. We may rewrite this local degree as a product of local degrees, yielding $\deg_p\sigma=\prod_{i=1}^n\deg_0 x^{d_i}$. By Remark~\ref{rem:rel orient- not all odd}, at least one of $d_1,\ldots,d_n$ is even. Since
\[\deg_0 ax^{d}=\begin{cases}\frac{d-1}{2}\cdot\mb{H}+\langle a\rangle & d\text{ odd},\\ \frac{d}{2}\cdot\mb{H} & d\text{ even},\end{cases}\]
we have that
\begin{align*}
\deg_p\sigma&=\prod_{d_i\text{ even}}\left(\tfrac{d_i}{2}\cdot\mb{H}\right)\cdot\prod_{d_i\text{ odd}}\left(\tfrac{d_i-1}{2}\cdot\mb{H}+\langle 1\rangle\right)\\
&=\left(\frac{\prod_\text{even} d_i}{2}\cdot\mb{H}\right)\left(\frac{(\prod_\text{odd}d_i)-1}{2}\cdot\mb{H}+\langle 1\rangle\right)\\
&=\frac{d_1\cdots d_n}{2}\cdot\mb{H}.
\end{align*}
Alternately, one can note that $\deg_p\sigma$ will be of the form $a\langle 1\rangle+b\langle -1\rangle$ for some $a,b\in\mb{Z}$. The values of $a$ and $b$ can then be determined by separately considering the rank and signature of $\deg_0 x^{d_i}$ for $1\leq i\leq n$. Either approach gives us $e(\mc{O}_{d_1,\ldots,d_n})=\frac{d_1\cdots d_n}{2}\cdot\mb{H}$.
\end{proof}

\begin{rem}
Theorem~\ref{thm:euler class-orientable} also follows from~\cite[Theorem 7.1]{Lev17} as described in~\cite[Proposition 19]{SW18}.
\end{rem}

In Section~\ref{sec:local degree}, we give a geometric interpretation of $\deg_p(f_1,\ldots,f_n)$. Paired with Theorem~\ref{thm:euler class-orientable}, this will prove Theorem~\ref{thm:main-orientable}.

\section{Formulas and geometric interpretations for the local degree}\label{sec:local degree}
In Section~\ref{sec:euler class}, we computed the Euler class of the vector bundle $\mc{O}_{d_1,\ldots,d_n}\to\mb{P}^n$. Roughly speaking, this Euler class equals the sum of local degrees over the intersection locus of $f_1,\ldots,f_n$. Once we provide a geometric interpretation of the local degree $\deg_p(f_1,\ldots,f_n)$, we will have an equation that counts geometric information concerning the intersection points of $f_1,\ldots,f_n$. This enumerative geometric equation will constitute our enriched version of B\'ezout's Theorem.

\subsection{Intersection multiplicity is the rank of the local degree}
We set out to prove that the intersection multiplicity $i_p(f_1,\ldots,f_n)$ of $f_1,\ldots,f_n$ at $p$ is the rank of the local degree $\deg_p(f_1,\ldots,f_n)$. We begin with a brief discussion of intersection multiplicity.

\begin{defn}\label{defn:intersection mult}\cite[Definition 4.1]{Sha13}
Given $p\in\mb{P}^n$ and hypersurfaces $f_1,\ldots,f_n$, the {\it intersection multiplicity} of $f_1,\ldots,f_n$ at $p$ is given by
\[i_p(f_1,\ldots,f_n)=\dim_k\mc{O}_{\mb{P}^n,p}/(f_1,\ldots,f_n).\]
\end{defn}

We next prove that this intersection multiplicity agrees with the rank of the local degree $\deg_p(f_1,\ldots,f_n)$. This will follow from the fact that both the intersection multiplicity and the local degree are given by local computations.

\begin{prop}\label{prop:rank=mult}
Let $f_1,\ldots,f_n$ be hypersurfaces in $\mb{P}^n$ that are all non-singular at a common intersection point $p$. Moreover, assume that $f_1,\ldots,f_n$ do not share a common component, so that $f_1\cap\cdots\cap f_n$ is a finite set. Then
\[\rank\deg_p(f_1,\ldots,f_n)=i_p(f_1,\ldots,f_n).\]
\end{prop}
\begin{proof}
Let $U$ be an affine neighborhood of $p$ with local coordinates $\vphi:U\to\mb{A}^n$. An algorithmic method for computing $\deg_p(f_1,\ldots,f_n)$ is outlined in~\cite[Table 1]{KW19}. In this method, we have that $\rank\deg_p(f_1,\ldots,f_n)=\dim_k\vphi_*(\mc{O}_{U,p}/(f_1,\ldots,f_n))$. (In the notation of~\cite{KW19}, $\vphi_*(\mc{O}_{U,p}/(f_1,\ldots,f_n))$ corresponds to $Q_p$. The assumption that $f_1\cap\ldots\cap f_n$ be a finite set ensures that $p$ is an isolated zero, so that $\dim_k Q_p$ is finite.) First, we have that $\mc{O}_{\mb{P}^n,p}/(f_1,\ldots,f_n)$ is isomorphic (as a $k$-algebra) to $\mc{O}_{U,p}/(f_1,\ldots,f_n)$. Since $\vphi:U\to\mb{A}^n$ is an isomorphism, we also have that $\mc{O}_{U,p}/(f_1,\ldots,f_n)$ and $\vphi_*(\mc{O}_{U,p}/(f_1,\ldots,f_n))$ are isomorphic as $k$-algebras. Thus $\rank\deg_p(f_1,\ldots,f_n)=i_p(f_1,\ldots,f_n)$, as desired.
\end{proof}

\subsection{Transverse intersections}
When the hypersurfaces $f_1,\ldots,f_n$ intersect transversely at a point $p$, the local degree $\deg_p(f_1,\ldots,f_n)$ has a geometric interpretation in terms of the gradient directions of each $f_i$ at $p$. Intuitively, $f_1,\ldots,f_n$ intersect transversely at $p$ as subschemes of $\mb{P}^n$ if their tangent spaces at $p$ overlap as little as possible. This idea can be made rigorous by the following definition.

\begin{defn}\cite[p. 18]{3264}
The subschemes $f_1,\ldots,f_n$ of $\mb{P}^n$ {\it intersect transversely at $p$} if each $f_i$ is smooth at $p$ and if $\codim(\bigcap_i T_p f_i)=\sum_i\codim T_p f_i$, where $\codim T_p f_i$ refers to the codimension of $T_p f_i$ as a subspace of the vector space $T_p\mb{P}^n_k\cong k(p)^n$.
\end{defn}

Transverse intersections of $n$ hypersurfaces in $\mb{P}^n$ are completely characterized by their intersection multiplicity.

\begin{prop}\label{prop:transverse=rank1}
The hypersurfaces $f_1,\ldots,f_n$ of $\mb{P}^n_k$ intersect transversely at a point $p$ if and only if $i_p(f_1,\ldots,f_n)=[k(p):k]$. In particular, if $p$ is a $k$-rational point, then $f_1,\ldots,f_n$ intersect transversely at $p$ if and only if $i_p(f_1,\ldots,f_n)=1$.
\end{prop}
\begin{proof}
Let $\mf{m}_p$ be the maximal ideal of $k[x_0,\ldots,x_n]$ corresponding to the point $p$. Since the intersection multiplicity $i_p(f_1,\ldots,f_n)$ is locally defined, we work in the local ring $k[x_0,\ldots,x_n]_{\mf{m}_p}=\mc{O}_{\mb{P}^n,p}$. The polynomials $f_1,\ldots,f_n$ are {\it local parameters} at $p$ in the sense of~\cite[Section 2.1]{Sha13} if and only if $(f_1,\ldots,f_n)=\mf{m}_p$ (see~\cite[Theorem 2.5]{Sha13}). By~\cite[Theorem 2.4]{Sha13}, we have $(f_1,\ldots,f_n)=\mf{m}_p$ if and only if $f_1,\ldots,f_n$ intersect transversely at $p$. Thus if $f_1,\ldots,f_n$ intersect transversely at $p$, then $(f_1,\ldots,f_n)=\mf{m}_p$ and hence $i_p(f_1,\ldots,f_n)=\dim_k\mc{O}_{\mb{P}^n,p}/\mf{m}_p=[k(p):k]$. On the other hand, the fact that $f_1,\ldots,f_n$ vanish at $p$ implies that $(f_1,\ldots,f_n)\subseteq\mf{m}_p$. Consequently, if $i_p(f_1,\ldots,f_n)=[k(p):k]=\dim_k\mc{O}_{\mb{P}^n,p}/\mf{m}_p$, then we have $(f_1,\ldots,f_n)=\mf{m}_p$ and hence $f_1,\ldots,f_n$ intersect transversely at $p$.
\end{proof}

By Proposition~\ref{prop:rank=mult}, it follows that $\rank\deg_p(f_1,\ldots,f_n)=[k(p):k]$ at transverse intersection points. As mentioned previously, $f_1,\ldots,f_n$ intersect transversely at $p$ if and only if $(f_1,\ldots,f_n)=\mf{m}_p$, which is equivalent to $p$ being a simple zero of the section $(f_1,\ldots,f_n)$. By a comment in~\cite[p. 17]{KW17}, the local degree $\deg_p(f_1,\ldots,f_n)$ at a simple zero $p$ is determined by the Jacobian of $f_1,\ldots,f_n$ (after locally trivializing). We make this precise below.

\begin{lem}\label{lem:deg of transverse}
Let $f_1,\ldots,f_n$ be hypersurfaces of $\mb{P}^n$ that intersect transversely at a point $p\in U_\ell$. To simplify notation, write $f^\ell_i:=f_i\circ\tilde{\vphi}_\ell^{-1}=f_i((-1)^\ell\cdot\frac{x_0}{x_\ell},\ldots,\frac{x_n}{x_\ell})$. Let
\[J_\ell=\det(\tfrac{\partial f^\ell_i}{\partial (x_j/x_\ell)})_{j\neq\ell},\] 
and let $\Tr_{k(p)/k}:\GW(k(p))\to\GW(k)$ be the trace on bilinear forms obtained by post-composing with the field trace $k(p)\to k$. Then
\[\deg_p(f_1,\ldots,f_n)=\Tr_{k(p)/k}\langle J_\ell(
\tilde{\vphi}_\ell(p))\rangle.\]
\end{lem}
\begin{proof}
By Lemma~\ref{lem:compatible-orientable case} and Lemma~\ref{lem:compatible-non-orientable}, the local trivialization $(f^\ell_1,\ldots,f^\ell_n)$ of our section $(f_1,\ldots,f_n)$ is compatible with our chosen relative orientation (relative to the effective Cartier divisor $D=\{x_0=0\})$ and Nisnevich coordinates $\{(U_i,\tilde{\vphi}_i)\}$. Since $p$ is a simple zero of $(f_1,\ldots,f_n)$, \cite[p. 17]{KW17} gives us that
\[\deg_p(f_1,\ldots,f_n)=\Tr_{k(p)/k}\langle J_\ell(\tilde{\vphi}_\ell(p))\rangle,\]
as desired.
\end{proof}

Moving forward, we will write $J_\ell(p)$ instead of $J_\ell(\tilde{\vphi}_\ell(p))$.
 
To obtain a geometric interpretation of the local degree at transverse intersections, we show how the Jacobian arises as a cross product of gradients. Working in one of our open affines, say $U_\ell\cong\mb{A}^n_k$, we let $e_j$ be the unit vector corresponding to the $\frac{x_j}{x_\ell}$-axis for all $j\neq\ell$. The gradient of $f^\ell_i$ is given by $\nabla f^\ell_i=\sum_{j\neq\ell}\frac{\partial f^\ell_i}{\partial (x_j/x_\ell)}\cdot e_j$.

\begin{defn}
Let $v_i=\sum_j a_{ij}\cdot e_j$ be a vector in $\mb{A}^n$, where $a_{ij}\in k$ and $e_j$ is as above. We may consider $\mb{A}^n$ as a subspace of $\mb{A}^{n+1}$, with a new unit vector $e_{n+1}:=e_1\times\cdots\times e_n$ corresponding to the direction perpendicular to $e_1,\ldots,e_n$. The {\it ($n$-ary) cross product} of $v_1,\ldots,v_n$ is a vector in the direction of $e_{n+1}$ given by
\[\bigtimes_{i=1}^nv_i=
\det\begin{pmatrix}
a_{11} & \cdots & a_{1n} & 0\\
\vdots & \ddots & \vdots & \vdots\\
a_{n1} & \cdots & a_{nn} & 0\\
e_1 & \cdots & e_n & e_{n+1}
\end{pmatrix}.\]
The dot product $(\bigtimes_{i=1}^n v_i)\cdot e_{n+1}$ is the signed volume of the parallelpiped bounded by $v_1,\ldots,v_n$. Note that this definition agrees with the more familiar notion of the cross product on $\mb{R}^3$.
\end{defn}

Under this definition, the Jacobian $J_\ell(p)$ is the value of the dot product $(\bigtimes_{i\neq\ell}(\nabla f^\ell_i(p)))\cdot e_{n+1}$, where $f^\ell_i(p)=f_i(\tilde{\vphi}_\ell(p))$. Thus the local degree at transverse intersections is described geometrically by the volume of the parallelpiped defined by the gradient vectors $\{\nabla f^\ell_i(p)\}_{i\neq\ell}$.

\subsection{Non-transverse intersections}\label{sec:non-transverse}
When our hypersurfaces $f_1,\ldots,f_n$ do not intersect transversely, the gradient directions of some $f_i$ and $f_j$ coincide. As a result, the parallelpiped spanned by the gradients of $f_1,\ldots,f_n$ has volume 0, so our previous geometric interpretation of $\deg_p(f_1,\ldots,f_n)$ no longer makes sense. We will discuss non-transverse intersections of pairs of curves in $\mb{P}^2$ and leave open the higher-dimensional case. Our goal is to reduce the calculation of the degree of two polynomials in two variables to the calculation of the degree of one power series in one variable. Let $f,g\in k[x,y]$ be polynomials of degrees $c$ and $d$, respectively. For simplicity, we will assume that $p=(0,0)$ is the origin and that $\frac{\partial g}{\partial y}(0,0)\neq 0$. We will also assume that $\Char{k}=0$, and we will discuss how to modify the following construction in positive characteristic.

\begin{lem}
Let $f$ and $g$ be curves in $\mb{P}^2$ that intersect at the origin, and assume that $\frac{\partial g}{\partial y}(0,0)\neq 0$. Then there exists some positive integer $n$ and some $a_n\in k^\times$ such that
\[\deg_0(f,g)=\begin{cases}\frac{n-1}{2}\cdot\mb{H}+\langle a_n\rangle & n\text{ odd},\\ \frac{n}{2}\cdot\mb{H} & n\text{ even}.\end{cases}\]
\end{lem}
\begin{proof}
By~\cite{KW19}, the local degree $\deg_0(f,g)$ may be computed by the bilinear form constructed in~\cite{SS75}. We will refer to this bilinear form as the Scheja--Storch form. As discussed on~\cite[p. 178]{SS75}, the Scheja--Storch form constructed for the local ring $\frac{k[x,y]_0}{(f,g)}$ is isomorphic to the Scheja--Storch form for the completion $\frac{k[[x,y]]}{(f,g)}$. We may thus work in $k[[x,y]]$ in order to compute $\deg_0(f,g)$.

Note that if $a\in k^\times$, we have $\deg_0(f,ag)=\langle a\rangle\cdot\deg_0(f,g)$. We may thus scale $g$ so that $\frac{\partial g}{\partial y}(0)=1$. By Hensel's Lemma, there exists a power series $G(x)\in k[[x]]$ such that $G(0)=0$ and  $g(x,G(x))=0$. We thus obtain an isomorphism
\[\begin{tikzcd}[row sep=0em]
\frac{k[[x,y]]}{(f,g)}\arrow[r,"\cong"] & \frac{k[[x]]}{(f(x,G(x)))}\\
y\arrow[r,mapsto,"h"]& G(x).
\end{tikzcd}\]

Intuitively, the $h$ transforms the curve $g(x,y)$ into our horizontal axis, and the curve $f(x,G(x))$ is the image of $f(x,y)$ under this transformation. In order for the isomorphism $h:\frac{k[[x,y]]}{(f,g)}\xrightarrow{\cong}\frac{k[[x]]}{(f(x,G(x)))}$ to preserve the local degree, it suffices to show that $h$ sends $\Jac(f,g)|_{x=y=0}$ to $\Jac(f(x,G(x)))|_{x=0}$. Indeed, the local degree is determined by the Scheja--Storch form~\cite{KW19}, and the Scheja--Storch form is determined by the Jacobian in characteristic 0~\cite[(4.7) Korollar]{SS75}. It follows that, given presentations of two local complete intersections $k[x_1,\ldots,x_m]_p/(r_1,\ldots,r_m)$ and $k[x_1,\ldots,x_n]_q/(s_1,\ldots,s_n)$ and an isomorphism $\phi:\frac{k[x_1,\ldots,x_m]_p}{(r_1,\ldots,r_m)}\to\frac{k[x_1,\ldots,x_n]_q}{(s_1,\ldots,s_n)}$, the bilinear form $\deg_p(r_1,\ldots,r_m)$ is isomorphic to the bilinear form $\deg_q(s_1,\ldots,s_n)$ if $\phi(\Jac(r_1,\ldots,r_m)(p))=\Jac(s_1,\ldots,s_n)(q)$.

To show that the bilinear forms $\deg_0(f,g)$ and $\deg_0(f(x,G(x)))$ are isomorphic, we thus need $h$ to send the Jacobian of $f$ and $g$ at $(0,0)$ to the derivative of $f(x,G(x))$ at $0$. That is, we need
\begin{align*}
h(\Jac(f,g))|_{x=y=0}&=h(\tfrac{\partial f}{\partial x}\cdot\tfrac{\partial g}{\partial y}-\tfrac{\partial f}{\partial y}\cdot\tfrac{\partial g}{\partial x})|_{x=y=0}\\
&=[f_x(x,G(x))\cdot g_y(x,G(x))-f_y(x,G(x))\cdot g_x(x,G(x))]_{x=y=0}\\
&=f_x(0,0)\cdot g_y(0,0)-f_y(0,0)\cdot g_x(0,0)
\end{align*}
to be equal to
\begin{align*}
\frac{d}{dx}f(x,G(x))|_{x=0}&=[f_x(x,G(x))+G'(x)\cdot f_y(x,G(x))]_{x=0}\\
&=f_x(0,0)+G'(0)\cdot f_y(0,0).
\end{align*}

Since $G(0)=0$, we have $g_y(x,G(x))|_{x=0}=g_y(0,0)$, which is equal to 1 by assumption. By the chain rule, we have $0=\frac{d}{dx}g(x,G(x))=g_x(x,G(x))+g_y(x,G(x))\cdot G'(x)$, so $g_x(x,G(x))=-g_y(x,G(x))\cdot G'(x)$. Thus $g_x(x,G(x))|_{x=0}=-G'(0)$, so $f_x(0,0)\cdot g_y(0,0)-f_y(0,0)\cdot g_x(0,0)=f_x(0,0)+G'(0)\cdot f_y(0,0)$, as desired.

Writing $f(x,G(x))=\sum_{i=n}^\infty a_ix^i=a_nx^n(1+\sum_{i=1}^\infty b_ix^i)$ with $a_n\neq 0$, we note that $1+\sum_{i=1}^\infty b_ix^i$ is a unit in $k[[x]]$ and hence the Scheja--Storch form of $\frac{k[[x]]}{(f(x,G(x)))}$ is equal to that of $\frac{k[[x]]}{(a_nx^n)}$. We thus have that $\deg_0(f,g)=\deg_0(a_nx^n)$, which is given by $\frac{n-1}{2}\cdot\mb{H}+\langle a_n\rangle$ if $n$ is odd and $\frac{n}{2}\cdot\mb{H}$ if $n$ is even.
\end{proof}

\begin{rem}
In any characteristic, the Scheja--Storch form is determined by a distinguished generator $E$ of the socle of $\frac{k[x,y]_0}{(f,g)}$. In order to modify the previous argument for the positive characteristic case, one would need to ensure that $q\circ h$ sends $E$ to the distinguished socle generator of $\frac{k[[x]]}{(f(x,G(x)))}$.
\end{rem}

\begin{prop}
We have that $n=i_0(f,g)$.
\end{prop}
\begin{proof}
By the above remarks and Proposition~\ref{prop:rank=mult}, both $n$ and $i_0(f,g)$ are equal to the rank of $\deg_0(f,g)$.
\end{proof}

This proposition allows us to completely understand the local degree $\deg_0(f,g)$ when $f$ and $g$ meet at the origin with even multiplicity. When $f$ and $g$ intersect with odd multiplicity, it remains to study the term $\langle a_n\rangle$. We discuss the geometric interpretation of $a_n$ over $\mb{R}$ in Lemma~\ref{lem:non-transverse over R}. We also give a recursive description of $a_n$ in terms of the coefficients of $f$ and $g$. Let
\[f=\sum_{i+j=0}^c f_{i,j}x^iy^j\qquad\text{and}\qquad g=\sum_{i+j=0}^d g_{i,j}x^iy^j.\]

We compute $a_n$ as the coefficient of $x^n=x^{i_0(f,g)}$ in $f(x,G(x))$. Thus $a_n=\sum_{i+j=n}f_{i,j}\cdot\gamma(j)$, where $\gamma(j)$ is the coefficient of $x^j$ in $G(x)^j$. But $\gamma(j)$ is equal to the coefficient of $x^j$ in $(G_0+G_1x+\ldots+G_jx^j)^j$, where $G(x)=\sum_{i=0}^\infty G_i x^i$. Expanding this product, we see that
\[\gamma(j)=\sum_{\substack{t_0+\ldots+t_j=j\\ \sum_u ut_u=j}}\binom{j}{t_0,\ldots,t_j}\prod_{u=0}^j G_u^{t_u},\]
where $\binom{j}{t_0,\ldots,t_j}$ denotes the multinomial coefficient. All that remains is to determine the coefficients $G_i$ of the power series $G(x)$. This is accomplished by repeatedly taking implicit derivatives. By assumption, $g(0,0)=0$, so we have $G_0=0$. Next, evaluating the partial derivative $\frac{\partial g}{\partial x}=\sum_{i+j=0}^d g_{i,j}(ix^{i-1}y^j+jx^iy^{j-1}\cdot\frac{\partial y}{\partial x})=0$ at $(0,0)$ gives us that $G_1=\frac{\partial y}{\partial x}(0,0)=-g_{1,0}/g_{0,1}$. Iterating this process allows us to compute $G_i=\frac{1}{i!}\cdot\frac{\partial^i y}{\partial x^i}(0,0)$.

\begin{rem}
In positive characteristic, we need an alternative form of the derivative in order to write down a Taylor series $G(x)$ under Hensel's Lemma. The Hasse derivative~\cite[p. 64]{Gol03} should be suitable for this purpose.
\end{rem}

\section{Specializations over some specific fields}\label{sec:examples}
For the following discussion, we let $N=-n-1+\sum_{i=1}^n d_i$ and assume that $N\equiv 0\mod 2$ so that $\mc{O}_{d_1,\ldots,d_n}\to\mb{P}^n$ is relatively orientable. Throughout this article, we have generally assumed that all intersections of $f_1,\ldots,f_n$ are transverse. This allows us to give better geometric interpretations of $\deg_p(f_1,\ldots,f_n)$. However, we also address non-transverse intersections in Sections~\ref{sec:over C} and \ref{sec:over R}. Our approach is as follows. For any field $k$, taking the rank of bilinear forms gives a homomorphism $\GW(k)\to\mb{Z}$. When $k$ is algebraically closed, $\rank:\GW(k)\xrightarrow{\cong}\mb{Z}$ is an isomorphism. When $k$ is not algebraically closed, we apply some other invariant to get a homomorphism of the form $\rank\times\op{invariant}:\GW(k)\to\mb{Z}\times G$ for some group $G$. The spirit of $\mb{A}^1$-enumerative geometry is that the $\mb{Z}$-valued count coming from the rank describes the geometric phenomena of classical enumerative geometry, while the additional $G$-valued count coming from the other invariant carries extra arithmetic-geometric information. This extra arithmetic-geometric information enriches the classical enumerative theorem when we work over a non-algebraically closed field.

\subsection{B\'ezout's Theorem over $\mb{C}$}\label{sec:over C}
Since $\mb{C}$ is algebraically closed, $\rank:\GW(\mb{C})\xrightarrow{\cong}\mb{Z}$ is an isomorphism. We thus recover B\'ezout's Theorem over $\mb{C}$ by applying $\rank$ to both sides of Equation~\ref{eq:main-orientable}. We know that $\rank\mb{H}=2$, so $\rank e(\mc{O}_{d_1,\ldots,d_n})=d_1\cdots d_n$. Moreover, we have $\rank\deg_p(f_1,\ldots,f_n)=i_p(f_1,\ldots,f_n)$ by Proposition~\ref{prop:rank=mult}. This gives us Equation~\ref{eq:bezout}, as expected.

\subsection{B\'ezout's Theorem over $\mb{R}$}\label{sec:over R}
We can represent any non-degenerate symmetric bilinear form over $\mb{R}$ by a diagonal matrix, where each diagonal entry is either $1,-1$, or $0$. By Sylvester's law of inertia, the isomorphism class of a bilinear form over $\mb{R}$ is determined by its rank and its signature, which we define to be the difference between the number of $1$s and the number of $-1$s on the diagonal. We thus have an isomorphism $\GW(\mb{R})\cong\mb{Z}\times\mb{Z}$ induced by $\rank\times\sign:\GW(\mb{R})\to\mb{Z}\times\mb{Z}$.

\begin{rem}
Since $|\sign(\beta)|\leq\rank(\beta)$, the homomorphism $\rank\times\sign:\GW(\mb{R})\to\mb{Z}\times\mb{Z}$ is not surjective. However, there is a group isomorphism $\GW(\mb{R})\cong\mb{Z}\times\mb{Z}$~\cite[Chapter II, Theorem 3.2 (4)]{Lam05}. Since $\rank\times\sign$ is injective, it follows that the image of $\rank\times\sign$ is isomorphic to $\mb{Z}\times\mb{Z}$.
\end{rem}

We obtain B\'ezout's Theorem over $\mb{R}$ by applying $\sign$ to both sides of Equation~\ref{eq:main-orientable}. Since $\sign\mb{H}=0$, we have that $\sign e(\mc{O}_{d_1,\ldots,d_n})=0$. When $f_1,\ldots,f_n$ intersect transversely at $p\in U_\ell$, the signature of $\deg_p(f_1,\ldots,f_n)$ is given by the sign of the volume of the parallelpiped defined by the gradient vectors $\{\nabla f^\ell_i(p)\}_{i\neq\ell}$. 

\begin{rem}\label{rem:complex-roots-sign-zero}
If $p$ is a non-rational intersection point of $f_1,\ldots,f_n$, then
\[\sign\deg_p(f_1,\ldots,f_n)=0.\]
To see this, suppose $\deg_p(f_1,\ldots,f_n)=\Tr_{\mb{C}/\mb{R}}\langle a+ib\rangle$. Since every element of $\mb{C}$ is a square, we have $\langle a+ib\rangle=\langle 1\rangle$. This bilinear form can thus be represented by
\[\begin{pmatrix}1&0\\0&-1\end{pmatrix},\]
which indeed has signature zero.
\end{rem}

By Remark~\ref{rem:complex-roots-sign-zero}, the local degree at non-real zeros does not contribute anything to our overall signed count $\sign{e(\mc{O}_{d_1,\ldots,d_n})}$. In particular, we only need to consider real zeros, so we may restrict our attention to the real points of $\mb{P}^n$ and the hypersurfaces $f_1,\ldots,f_n$. This allows us to apply Milnor's local alteration trick~\cite[\S 6, Step 2]{Mil65} from real differential topology.

\begin{lem}\label{lem:non-transverse over R}
Let $\sigma\in\mc{O}_{d_1,\ldots,d_n}$ be a section corresponding to the hypersurfaces $f_1,\ldots,f_n$ in $\mb{P}^n_{\mb{R}}$ that intersect at a rational point $p$, and let $U$ be an open neighborhood (in the real topology) about $p$ that does not contain any other zeros of $\sigma$. Then there exist open neighborhoods $U'\subseteq U$ about $p$ and $V\subseteq H^0(\mb{P}^n_{\mb{R}},\mc{O}_{d_1,\ldots,d_n})$ about $\sigma$ such that $\sign\deg_p\sigma=\sum_{q\in U'}\sign\deg_q\sigma'$ for all $\sigma'\in V$.
\end{lem}
\begin{proof}
This follows from the proof of~\cite[Lemma 2.4]{LV19}. We recall the relevant details for the reader's convenience. Let $\vphi:\mc{O}_{d_1,\ldots,d_n}|_U\xrightarrow{\cong}\mb{R}^n$ and $\psi:\mc{T}U\xrightarrow{\cong}\mb{R}^n$ be local trivializations such that $\det(\vphi^{-1}\circ\psi)$ is a square in $\mc{O}(-n-1+\sum_{i=1}^n d_i)$ under the chosen relative orientation. On $U$, we can equate sections $\sigma\in\mc{O}_{d_1,\ldots,d_n}|_U$ with vector fields $v_\sigma=\psi^{-1}\circ\vphi(\sigma)$ on $U$, which allows us to use Milnor's local alteration trick as follows. Let $U''\subset U'\subset U$ be sufficiently small (real) open neighborhoods about $p$, and let $\lambda:U\to[0,1]$ be a smooth bump function such that $\lambda|_{U''}=1$ and $\lambda|_{U\backslash U'}=0$. Taking a sufficiently small regular value $y$ of $v_\sigma$, the vector field $v(x)=v_\sigma(x)-\lambda(x)y$ is non-degenerate on $U'$. Let $\iota_q w$ denote Milnor's local index of a vector field $w$ at a zero $q$. By~\cite[\S 6, Theorem 1]{Mil65}, we have that $\sum_{q\in U'}\iota_q v$ is equal to the degree of the Gauss map $\bar{v}:\partial U'\to S^{n-1}$. The degree of the Gauss map, and hence the sum $\sum_{q\in U'}\iota_q v$, is continuous (and thus locally constant) in $\sigma$. Let $V\subseteq H^0(\mb{P}^n_\mb{R},\mc{O}_{d_1,\ldots,d_n})$ be a sufficiently small neighborhood about $\sigma$; in particular, the zeros of the altered vector field $v'$ corresponding to the section $\sigma'$ should remain in $U'$ as $\sigma'$ varies. Then for all $\sigma'\in V$, we have that $\iota_p v=\sum_{q\in U'}\iota_q v'$. Finally, we remark that $\iota_p v=\sign\deg_p\sigma$, as was proved by Eisenbud and Levine~\cite[Main Theorem]{Eis78}.
\end{proof}

As an aside, this proof implies that the degree of the Gauss map associated to the hypersurfaces $f_1,\ldots,f_n$ at an intersection point $p$ is bounded by the intersection multiplicity $i_p(f_1,\ldots,f_n)$. We state this as a proposition.

\begin{prop}
Let $\sigma$ and $v$ be as in Lemma~\ref{lem:non-transverse over R}, and let $\bar{v}$ be the corresponding Gauss map. Then $|\deg\bar{v}|\leq i_p(f_1,\ldots,f_n)$.
\end{prop}
\begin{proof}
By~\cite[\S 6, Theorem 1]{Mil65} and ~\cite[Main Theorem]{Eis78}, we have that $\deg\bar{v}=\sign\deg_p\sigma$. Since $|\sign\deg_p\sigma|\leq\rank\deg_p\sigma$, Proposition~\ref{prop:rank=mult} implies that $|\deg\bar{v}|\leq i_p(f_1,\ldots,f_n)$.
\end{proof}

\begin{rem}
By Lemma~\ref{lem:non-transverse over R}, we can compute the crossing sign of a non-transverse intersection by slightly perturbing our chosen section. Since generic intersections are transverse, we may choose our new section to have only transverse intersections. As a consequence, the crossing sign of a non-transverse intersection is given by a sum of crossing signs of transverse intersections. This is illustrated in Figure~\ref{fig:crossing sign}.
\end{rem}

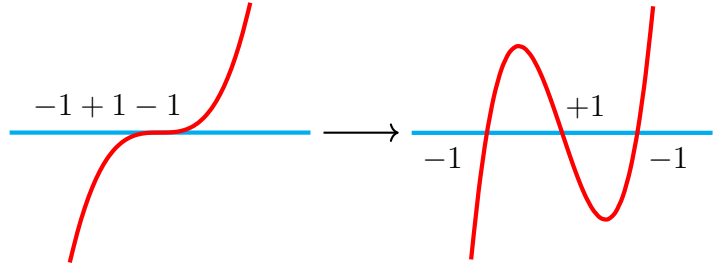
\begin{figure}[h]
\centering
\begin{tikzpicture}
\node (a) at (0,0)
{
	\begin{tikzpicture}
      \draw [cyan,ultra thick] (-2,0) -- (2,0);
      \draw[domain=-1.2:1.2,smooth,variable=\x,red,ultra thick]  plot ({\x},{\x*\x*\x});
      \node at (-0.7,0.35) {$-1+1-1$};
	\end{tikzpicture}
};
\node (b) at (a.east) [anchor=west,xshift=1cm]
{
\begin{tikzpicture}
      \draw [cyan,ultra thick] (-2,0) -- (2,0);
      \draw[domain=-1.21:1.21,smooth,variable=\x,red,ultra thick]  plot ({\x},{3*\x*(\x-1)*(\x+1)});
      \node at (-2,-0.35) {$-1$};
      \node at (-0.1,0.35) {$+1$};
      \node at (1,-0.35) {$-1$};
\end{tikzpicture}
};
\draw [->,thick] (a)--(b);
\end{tikzpicture}
\caption{The crossing sign at a non-transverse intersection.}\label{fig:crossing sign}
\end{figure}

Perturbing our hypersurfaces to ensure that they intersect transversely, we may thus call $\sign\deg_p(f_1,\ldots,f_n)$ the {\it crossing sign} of $f_1,\ldots,f_n$ at $p$, and we obtain the following theorem.

\begin{thm}[B\'ezout's Theorem over $\mb{R}$]\label{thm:bezout-over-R}
Let $f_1,\ldots,f_n$ be hypersurfaces in $\mb{P}^n_\mb{R}$, and let $d_i$ be the degree of $f_i$ for each $i$. Assume that $f_1,\ldots,f_n$ have no common components and that $-n-1+\sum_{i=1}^n d_i\equiv 0\mod 2$. Then, summing over the intersection points of $f_1,\ldots,f_n$, there are an equal number of positive and negative crossings of $f_1,\ldots,f_n$.
\end{thm}

\begin{ex}
We can now make sense of the problematic conic and cubic in $\mb{P}^2_\mb{R}$ from Figure~\ref{fig:conic-cubic}. To be specific, let $f_1=x_0^2x_2-x_1^3$ and $f_2=x_1^2+x_2^2-2x_0^2$. The only intersection points of $f_1$ and $f_2$ are $p_1=\pcoor{1:-1:-1}$ and $p_2=\pcoor{1:1:1}$. The crossing sign of $f_1$ and $f_2$ at $p_i$ is given by the right hand rule on the gradient vectors of $f_1^0$ and $f_2^0$ at $\vphi_0(p_i)$. We now demonstrate this calculation. Let $e_i$ be the the unit basis vector along the $(\tfrac{x_i}{x_0})$-axis, and let $e_3:=e_1\times e_2$. Then $\nabla f_1^0=-3(\tfrac{x_1}{x_0})^2\cdot e_1+e_2$ and $\nabla f_2^0=2(\tfrac{x_1}{x_0})\cdot e_1+2(\tfrac{x_2}{x_0})\cdot e_2$, so $\nabla f_1^0\times\nabla f_2^0=(-6(\tfrac{x_1}{x_0})^2(\tfrac{x_2}{x_0})-2(\tfrac{x_1}{x_0}))\cdot e_3$. The crossing sign at $p_i$ is computed by taking the sign of the dot product of $\nabla f_1^0\times\nabla f_2^0(\vphi_0(p_i))$ and $e_3$. This gives us
\begin{align*}
\sign(\nabla f_1^0\times\nabla f_2^0(\vphi(p_1))\cdot e_3)&=\sign(6+2)=1,\\
\sign(\nabla f_1^0\times\nabla f_2^0(\vphi(p_2))\cdot e_3)&=\sign(-6-2)=-1.
\end{align*}
We illustrate this calculation in Figure~\ref{fig:bezout-over-R}.
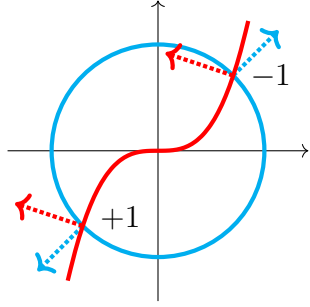
\begin{figure}[h]
\centering
\begin{tikzpicture}
      \draw[->] (-2,0) -- (2,0);
      \draw[->] (0,-2) -- (0,2);
      \draw [cyan,ultra thick] (0,0) circle [radius=1.4142];
      \draw[->,cyan,ultra thick,densely dotted] (-1,-1) -- (-1.6,-1.6);
      \draw[->,cyan,ultra thick,densely dotted] (1,1) -- (1.6,1.6);
      \draw[domain=-1.2:1.2,smooth,variable=\x,red,ultra thick]  plot ({\x},{\x*\x*\x});
      \draw[->,red,ultra thick,densely dotted] (-1,-1) -- (-1.9,-0.7);
      \draw[->,red,ultra thick,densely dotted] (1,1) -- (0.1,1.3);
      \draw (-.5,-.9) node{$+1$};
      \draw (1.5,1) node{$-1$};
\end{tikzpicture}
\caption{A conic and a cubic over $\mb{R}$.}\label{fig:bezout-over-R}
\end{figure}
\end{ex}

\subsection{B\'ezout's Theorem over finite fields}
Let $\mb{F}_q$ be the finite field of order $q$. Over $\mb{F}_q$, non-degenerate symmetric bilinear forms are classified up to isomorphism by their rank and discriminant~\cite[Chapter II, Theorem 3.5 (1)]{Lam05}, so we have an isomorphism $\rank\times\disc:\GW(\mb{F}_q)\xrightarrow{\cong}\mb{Z}\times\mb{F}_q^\times/(\mb{F}_q^\times)^2$. For simplicity, we assume that $\mb{F}_q$ has odd characteristic, so that $q$ is the power of some odd prime. As a result, we have $\mb{F}_q^\times/(\mb{F}_q^\times)^2\cong\mb{Z}/2\mb{Z}$, and we can classify elements of a given rank in $\GW(\mb{F}_q)$ by whether or not their discriminant is a square. Taking the discriminant of Equation~\ref{eq:main-orientable}, we get $\disc e(\mc{O}_{d_1,\ldots,d_n})=(-1)^{d_1\cdots d_n/2}$. This makes sense, since the fraction $\frac{d_1\cdots d_n}{2}$ is an integer by Remark~\ref{rem:rel orient- not all odd}. We note that $\disc e(\mc{O}_{d_1,\ldots,d_n})$ is a perfect square in $\mb{F}_q$ if and only if $\frac{d_1\cdots d_n}{2}$ is even or $q\equiv 1\mod 4$. The left hand side of Equation~\ref{eq:main-orientable} gives us
\[\prod_{\text{points}}\disc\deg_p(f_1,\ldots,f_n).\]

This product is a perfect square in $\mb{F}_q$ if and only if there are an even number of intersection points $p$ such that $\disc\deg_p(f_1,\ldots,f_n)$ is not a square. It thus remains to determine when $\disc\deg_p(f_1,\ldots,f_n)$ is a square in $\mb{F}_q$. Let $\mb{F}_{q^b}$ be the field of definition of a transverse intersection point $p\in U_\ell$. By~\cite[Section II.2]{CP84}, we have that $\disc\Tr_{\mb{F}_{q^b}/\mb{F}_q}\langle J_\ell(p)\rangle=\norm(J_\ell(p))\cdot\disc\Tr_{\mb{F}_{q^b}/\mb{F}_q}\langle 1\rangle$. Since $\norm:\mb{F}_{q^b}^\times\to\mb{F}_q^\times$ is a homomorphism, $\norm$ takes squares in $\mb{F}_{q^b}^\times$ to squares in $\mb{F}_q^\times$. On the other hand, Hilbert's Theorem 90 implies that $\norm:\mb{F}_{q^b}^\times\to\mb{F}_q^\times$ is surjective, so if $\norm(J_\ell(p))$ is a square, then we have $\norm(J_\ell(p))=\norm(y^2)$ for some $y\in\mb{F}_{q^b}^\times$. Hilbert's Theorem 90 also implies that $\norm$ has kernel $\{z^{q-1}:z\in\mb{F}_{q^b}^\times\}$, so there exists some $z\in\mb{F}_{q^b}^\times$ such that $J_\ell(p)=y^2z^{q-1}$. It follows that $J_\ell(p)=(yz^{(q-1)/2})^2$, so $J_\ell(p)$ is a square if and only if $\norm(J_\ell(p))$ is a square. 

\begin{defn}
We will call a transverse intersection point $p$ of $f_1,\ldots,f_n$ a {\it positive} intersection point if $J_\ell(p)$ is a square in $\mb{F}_{q^b}^\times$. We call $p$ a {\it negative} intersection point if $J_\ell(p)$ is not a square in $\mb{F}_{q^b}^\times$.
\end{defn}

To characterize when $\Tr_{\mb{F}_{q^b}/\mb{F}_q}\langle 1\rangle$ is a square, we let $\alpha$ be a primitive element of the extension $\mb{F}_{q^b}/\mb{F}_q$, so that $\{1,\alpha,\ldots,\alpha^{b-1}\}$ is an $\mb{F}_q$-basis for $\mb{F}_{q^b}$. The matrix representing $\Tr_{\mb{F}_{q^b}/\mb{F}_q}\langle 1\rangle$ with respect to this basis is $MM^T$, where
\[M=\begin{pmatrix}
1 & 1 & \cdots & 1\\ 
\alpha & \alpha^q & \cdots & \alpha^{q^{b-1}}\\
\alpha^2 & \alpha^{2q} & \cdots & \alpha^{2q^{b-1}}\\
\vdots & \vdots & \ddots & \vdots\\
\alpha^{b-1} & \alpha^{(b-1)q} & \cdots & \alpha^{(b-1)q^{b-1}}
\end{pmatrix}.\]
Indeed, the $(i,j)^\textsuperscript{th}$ entry of $MM^T$ is
\begin{align*}
\sum_{\ell=0}^{b-1}\alpha^{(i-1)q^\ell}\alpha^{(j-1)q^\ell}&=\sum_{\ell=0}^{b-1}(\alpha^{i+j-2})^{q^\ell}\\
&=\Tr_{\mb{F}_{q^b}/\mb{F}_q}(\alpha^{i+j-2}).
\end{align*}
By definition of the trace form, this is equal to the $(i,j)^\textsuperscript{th}$ entry of the Gram matrix of $\Tr_{\mb{F}_{q^b}/\mb{F}_q}\langle 1\rangle$ with respect to the basis $\{1,\alpha,\ldots,\alpha^{b-1}\}$. Thus we have
\begin{align*}
\disc\Tr_{\mb{F}_{q^b}/\mb{F}_q}\langle 1\rangle&=\det(MM^T)=(\det{M})^2\\
&=\prod_{0\leq i<j\leq b-1}(\alpha^{q^j}-\alpha^{q^i})^2.
\end{align*}
So $\disc\Tr_{\mb{F}_{q^b}/\mb{F}_q}\langle 1\rangle$ is a square in $\mb{F}_q$ if and only if $\delta=\prod_{i<j}(\alpha^{q^j}-\alpha^{q^i})$ is an element of $\mb{F}_q$. Since $\Gal(\mb{F}_{q^b}/\mb{F}_q)$ is cyclic and generated by the Frobenius $F$, we have that $\delta\in\mb{F}_q$ if and only if $F(\delta)=\delta$. But $F(\delta)=\eps\delta$, where $\eps$ is the sign of the permutation $(1\ 2\ \ldots\ b-1)$. We know that $\eps=1$ if $b$ is odd and $\eps=-1$ if $b$ is even, so $\disc\Tr_{\mb{F}_{q^b}/\mb{F}_q}\langle 1\rangle$ is a square in $\mb{F}_q$ if and only if $b$ is odd. Summing all this together, we have
\[\disc\Tr_{\mb{F}_{q^b}/\mb{F}_q}\langle J_\ell(p)\rangle=\begin{cases}
\text{a square} & \text{if }p\text{ is a positive intersection and }b\text{ is odd,}\\
\text{a square} & \text{if }p\text{ is a negative intersection and }b\text{ is even,}\\
\text{a non-square} & \text{if }p\text{ is a positive intersection and }b\text{ is even,}\\
\text{a non-square} & \text{if }p\text{ is a negative intersection and }b\text{ is odd.}
\end{cases}\]

We summarize this information in the following theorem.

\begin{thm}[B\'ezout's Theorem over $\mb{F}_q$]
Let $f_1,\ldots,f_n$ be hypersurfaces in $\mb{P}^n_{\mb{F}_q}$, and let $d_i$ be the degree of $f_i$ for each $i$. Assume that $f_1,\ldots,f_n$ intersect transversely and that $-n-1+\sum_{i=1}^n d_i\equiv 0\mod 2$.
\begin{enumerate}[(a)]
\item If $\frac{d_1\cdots d_n}{2}$ is even or $q\equiv 1\mod 4$, then
\begin{align*}
&\#\ \text{positive intersections with field of definition }\mb{F}_{q^b}\text{ for }b\text{ even}\\
+\ &\#\ \text{negative intersections with field of definition }\mb{F}_{q^b}\text{ for }b\text{ odd}\equiv 0\mod 2.
\end{align*}
\item If $\frac{d_1\cdots d_n}{2}$ is odd and $q\not\equiv 1\mod 4$, then
\begin{align*}
&\#\ \text{positive intersections with field of definition }\mb{F}_{q^b}\text{ for }b\text{ even}\\
+\ &\#\ \text{negative intersections with field of definition }\mb{F}_{q^b}\text{ for }b\text{ odd}\equiv 1\mod 2.
\end{align*}
\end{enumerate}
\end{thm}

\subsection{B\'ezout's Theorem over $\mb{C}((t))$}
We begin by describing $\GW(\mb{C}((t)))$. The field $\mb{C}((t))$ of Laurent series consists of elements of the form $g=\sum_{i=m}^\infty a_it^i$, where $m\in\mb{Z}$ and $a_m\neq 0$, and of the element 0. With the valuation $v(g)=m$, the pair $(\mb{C}((t)),v)$ is a {\it complete discretely valuated field} (see e.g.~\cite[Chapter VI, Section 1]{Lam05}). By slight abuse of terminology, we call the non-zero elements of $\mb{C}((t))$ with $v(g)=0$ {\it units}. Units in $\mb{C}((t))$ are of the form $g=\sum_{i=0}^\infty a_it^i$. Fixing $t$ as our uniformizer, every non-zero element of $\mb{C}((t))$ can be written as $g=ut^{v(g)}$ for some unit $u$. By relation (i) of Section~\ref{sec:gw}, it follows that $\langle g\rangle=\langle u\rangle$ if $v(g)$ is even and $\langle g\rangle=\langle ut\rangle$ if $v(g)$ is odd. By~\cite[Chapter VI, Lemma 1.1]{Lam05}, a unit $u$ is a square in $\mb{C}((t))$ if and only if $u(0)=a_0$ is a square in $\mb{C}$. Since $\mb{C}$ is algebraically closed, we have $\langle g\rangle=\langle 1\rangle$ if $v(g)$ is even and $\langle g\rangle=\langle t\rangle$ if $v(g)$ is odd. It follows that $\GW(\mb{C}((t)))$ is generated by $\langle 1\rangle$ and $\langle t\rangle$. Since $\langle 1\rangle=\langle -1\rangle$ and $\langle t\rangle=\langle -t\rangle$, we have that $2\langle 1\rangle=2\langle t\rangle=\mb{H}$. We thus get a well-defined group isomorphism
\[\GW(\mb{C}((t)))\cong\frac{\mb{Z}[\langle t\rangle]}{(\langle t\rangle^2-1,2\langle t\rangle-2)}\cong\mb{Z}\times\mb{Z}/2\mb{Z}.\]
We would like to realize the isomorphism $\GW(\mb{C}((t)))\cong\mb{Z}\times\mb{Z}/2\mb{Z}$ as a map of the form $\rank\times\op{invariant}$. 

\begin{prop}
Let $\rank(m\langle 1\rangle+n\langle t\rangle)=m+n$, and let $\disc(m\langle 1\rangle+n\langle t\rangle)=n\mod 2$. Then $\rank\times\disc:\GW(\mb{C}((t)))\xrightarrow{\cong}\mb{Z}\times\mb{Z}/2\mb{Z}$ is an isomorphism of abelian groups.
\end{prop}
\begin{proof}
It can readily be checked that $\rank\times\disc$ is a well-defined group homomorphism. Moreover, $\rank\times\disc$ is surjective, since $\rank\times\disc(m\langle 1\rangle)=(m,0)$ and $\rank\times\disc((m-1)\langle 1\rangle+\langle t\rangle)=(m,1)$ for all $m\in\mb{Z}$. Finally, $rank\times\disc$ is injective. Indeed, if $\rank\times\disc(m\langle 1\rangle+n\langle t\rangle)=(0,0)$, then $m+n=0$ and $n\equiv 0\mod 2$. Thus $m=-n=-2s$ for some $s\in\mb{Z}$, so $m\langle 1\rangle+n\langle t\rangle=s(2\langle t\rangle-2\langle 1\rangle)=0$.
\end{proof}

\begin{rem}
The homomorphism $\disc:\GW(\mb{C}((t)))\to\mb{Z}/2\mb{Z}$ is the traditional discriminant, simply written additively.
\end{rem}

We may now obtain B\'ezout's Theorem over $\mb{C}((t))$ by applying $\disc$ to Equation~\ref{eq:main-orientable}. Since $\mb{H}=2\cdot\langle 1\rangle=2\cdot\langle t\rangle$, we have $\disc\mb{H}=0$ and hence $\disc{e(\mc{O}_{d_1,\ldots,d_n})}=0$. This will be equal to $\disc(\sum\deg_p(f_1,\ldots,f_n))=\sum\disc\deg_p(f_1,\ldots,f_n)$, so we need to understand $\disc\deg_p(f_1,\ldots,f_n)$. Any degree $m$ extension of $\mb{C}((t))$ is a cyclic Galois extension of the form $\mb{C}((t^{1/m}))$~\cite[XIII.2, p. 191]{Ser79}. By our previous discussion, $\GW(\mb{C}((t^{1/m})))$ is generated by $\langle 1\rangle$ and $\langle t^{1/m}\rangle$, so at transverse intersection points, $\deg_p(f_1,\ldots,f_n)$ is either of the form $\Tr_{\mb{C}((t^{1/m}))/\mb{C}((t))}\langle 1\rangle$ or $\Tr_{\mb{C}((t^{1/m}))/\mb{C}((t))}\langle t^{1/m}\rangle$.

\begin{defn}
In analogy with the finite field case, we call a transverse intersection point $p$ of $f_1,\ldots,f_n$ a {\it positive} intersection point if $\deg_p(f_1,\ldots,f_n)=\Tr_{\mb{C}((t^{1/m}))/\mb{C}((t))}\langle 1\rangle$. We call $p$ a {\it negative} intersection point if $\deg_p(f_1,\ldots,f_n)=\Tr_{\mb{C}((t^{1/m}))/\mb{C}((t))}\langle t^{1/m}\rangle$.
\end{defn}

\begin{lem}
If $m$ is a positive integer, then we have $\disc\Tr_{\mb{C}((t^{1/m}))/\mb{C}((t))}\langle 1\rangle\equiv m-1\mod 2$ and $\disc\Tr_{\mb{C}((t^{1/m}))/\mb{C}((t))}\langle t^{1/m}\rangle\equiv m\mod 2$.
\end{lem}
\begin{proof}
Mirroring the case of finite fields, we let $t^{1/m}$ be our primitive element. The Galois group of $\mb{C}((t^{1/m}))$ over $\mb{C}((t))$ is generated by $\vphi:t^{1/m}\mapsto\zeta t^{1/m}$, where $\zeta=e^{2\pi i/m}$ is a primitive $m$\textsuperscript{th} root of unity. We have the $\mb{C}((t))$-basis $\{1,t^{1/m},\ldots,t^{(m-1)/m}\}$ for $\mb{C}((t^{1/m}))$. The Gram matrix for $\Tr_{\mb{C}((t^{1/m}))/\mb{C}((t))}\langle u\rangle$ with respect to this basis is given by the product $AB$, where $a_{ij}=\vphi^{j-1}(t^{(i-1)/m})$ and $b_{ij}=\vphi^{i-1}(ut^{(j-1)/m})$ are the entries in the $i$\textsuperscript{th} row and $j$\textsuperscript{th} column of $A$ and $B$, respectively. When $u=1$, this product of matrices has entries
\[c_{ij}=t^{(i+j-2)/m}\sum_{\ell=0}^{m-1}\zeta^{\ell(i+j-2)}=\begin{cases}0 & m\nmid i+j-2,\\ mt^{(i+j-2)/m} & m\mid i+j-2.\end{cases}\]

As a result, we have
\begin{align*}
\Tr_{\mb{C}((t^{1/m}))/\mb{C}((t))}\langle 1\rangle=\begin{pmatrix}m & 0 & \cdots & 0\\ 0 & 0 & \cdots & mt\\ \vdots & \vdots & \reflectbox{$\ddots$} & \vdots\\ 0 & mt & \cdots & 0\end{pmatrix}=(m-1)\cdot\langle t\rangle+\langle 1\rangle.
\end{align*}

Thus $\disc\Tr_{\mb{C}((t^{1/m}))/\mb{C}((t))}\langle 1\rangle\equiv m-1\mod 2$. When $u=t^{1/m}$, the product $AB$ has entries
\[c_{ij}=t^{(i+j-1)/m}\sum_{\ell=0}^{m-1}\zeta^{\ell(i+j-1)}=\begin{cases}0 & m\nmid i+j-1,\\ mt^{(i+j-1)/m} & m\mid i+j-1.\end{cases}\]

As a result, we have
\begin{align*}
\Tr_{\mb{C}((t^{1/m}))/\mb{C}((t))}\langle t^{1/m}\rangle=\begin{pmatrix}0 & \cdots & 0 & mt\\ 0 & \cdots & mt & 0\\ \vdots & \reflectbox{$\ddots$} & \vdots & \vdots\\ mt & \cdots & 0 & 0\end{pmatrix}=m\cdot\langle t\rangle.
\end{align*}

Thus $\disc\Tr_{\mb{C}((t^{1/m}))/\mb{C}((t))}\langle t^{1/m}\rangle\equiv m\mod 2$.
\end{proof}

As a result, we have proved the following theorem.

\begin{thm}[B\'ezout's Theorem over $\mb{C}((t))$]
Let $f_1,\ldots,f_n$ be hypersurfaces in $\mb{P}^n_{\mb{C}((t))}$, and let $d_i$ be the degree of $f_i$ for each $i$. Assume that $f_1,\ldots,f_n$ intersect transversely and that $-n-1+\sum_{i=1}^n d_i\equiv 0\mod 2$. Then
\begin{align*}
&\#\ \text{positive intersections with field of definition }\mb{C}((t^{1/m}))\text{ for }m\text{ even}\\
+\ &\#\ \text{negative intersections with field of definition }\mb{C}((t^{1/m}))\text{ for }m\text{ odd}\equiv 0\mod 2.
\end{align*}
\end{thm}

\subsection{B\'ezout's Theorem over $\mb{Q}$}
In contrast with the previous fields we have considered, we need several invariants to understand $\GW(\mb{Q})$. Letting $\W(\mb{Q})$ denote the Witt ring of $\mb{Q}$, we have
\[\begin{tikzcd}
\GW(\mb{Q}) \arrow[r,"\rank"] \arrow[d,swap,"\mr{mod}\ \mb{H}"] & \mb{Z} \arrow[d,"\mr{mod}\ 2"]\\ \W(\mb{Q})\arrow[r] &\mb{Z}/2\mb{Z}.
\end{tikzcd}\]
This gives us an isomorphism $\rank\times\ \mathrm{mod}\ \mb{H}:\GW(\mb{Q})\xrightarrow{\cong}\mb{Z}\times_{\mb{Z}/2\mb{Z}}\W(\mb{Q})$. In order to obtain B\'ezout's Theorem over $\mb{Q}$, it thus suffices to describe $\W(\mb{Q})$. By the weak Hasse-Minkowski principle (see e.g.~\cite[Chapter VI, Section 4, (4.4)]{Lam05}), we have an isomorphism
\[\partial:=\sign\times\ \partial_2\times\oplus\partial_p:\W(\mb{Q})\xrightarrow{\cong}\mb{Z}\times\mb{Z}/2\mb{Z}\times\bigoplus_{p\neq 2}\W(\mb{F}_p).\]
When $p\equiv 3\mod 4$, we have $\W(\mb{F}_p)\cong\mb{Z}/4\mb{Z}$, generated by $\langle 1\rangle$. When $p\equiv 1\mod 4$, the Witt ring $\W(\mb{F}_p)$ is isomorphic to the group ring $(\mb{Z}/2\mb{Z})[\mb{F}_p^\times/(\mb{F}_p^\times)^2]$, whose underlying group structure is isomorphic to $\mb{Z}/2\mb{Z}\times\mb{Z}/2\mb{Z}$, generated by $\langle 1\rangle$ and $\langle r\rangle$, where $r$ is a non-square in $\mb{F}_p^\times$. The invariant $\partial_2$ is given by $\partial_2(\beta)=v_2(\disc\beta)\mod 2$, where $v_2$ is the 2-adic valuation. For any odd prime $p$, any element of $\mb{Q}$ may be written as $q=up^{v_p(q)}$, where $v_p$ is the $p$-adic valuation and $v_p(u)=0$. It follows that $\langle q\rangle=\langle u\rangle$ if $v_p(q)$ is even and $\langle q\rangle=\langle up\rangle$ if $v_p(q)$ is odd. We define $\partial_p$ by setting $\partial_p\langle u\rangle=0$ for any $p$-adic unit $u$ and $\partial_p\langle up\rangle=\langle\bar{u}\rangle$, where $\bar{u}$ is the image of $u$ under the composition $\mb{Q}\inj\mb{Q}_p\to\mb{F}_p$. Explicitly, if $q$ is a non-zero rational number with $v_p(q)$ odd, then we write $q=\frac{a}{b}\cdot p^{v_p(q)}$ with $a$ and $b$ prime to $p$. Since $\langle q\rangle=\langle\frac{a}{b}\cdot p\rangle$ in $\W(\mb{Q})$, our definition for $\partial_p$ gives us $\partial_p\langle q\rangle=\langle(a\mod p)(b\mod p)^{-1}\rangle$ in $\W(\mb{F}_p)$. We obtain B\'ezout's Theorem over $\mb{Q}$ by applying $\partial\circ\mr{mod}\ \mb{H}$ to both sides of Equation~\ref{eq:main-orientable}. Since $e(\mc{O}_{d_1,\ldots,d_n})\in\GW(\mb{Q})$ is a multiple of $\mb{H}$, it has trivial image in $\W(\mb{Q})$. Thus the right hand side of Equation~\ref{eq:main-orientable} becomes $(0,0,0,\ldots)$, while the left hand side becomes $\sum_\text{points}\partial(\deg_x(f_1,\ldots,f_n)\mod\mb{H})$. In summary, we have the following theorem.

\begin{thm}[B\'ezout's Theorem over $\mb{Q}$]\label{thm:over Q}
Let $f_1,\ldots,f_n$ be hypersurfaces in $\mb{P}^n_{\mb{Q}}$, and let $d_i$ be the degree of $f_i$ for each $i$. Assume that $f_1,\ldots,f_n$ intersect transversely and that $-n-1+\sum_{i=1}^n d_i\equiv 0\mod 2$. Then we have the following statements.
\begin{enumerate}[(a)]
\item We have $\sum_x\sign\deg_x(f_1,\ldots,f_n)=0$.
\item We have $\sum_x\partial_2\deg_x(f_1,\ldots,f_n)\equiv 0\mod 2$.
\item For each prime $p\equiv 3\mod 4$, we have $\sum_x\partial_p\deg_x(f_1,\ldots,f_n)\equiv 0\mod 4$. (Here, we identify $\partial_p\deg_x(f_1,\ldots,f_n)$ with its image in $\mb{Z}/4\mb{Z}$.)
\item For each prime $p\equiv 1\mod 4$, we have $\sum_x\partial_p\deg_x(f_1,\ldots,f_n)\equiv(0,0)\mod(2,2)$. (Here, we identify $\partial_p\deg_x(f_1,\ldots,f_n)$ with its image in $\mb{Z}/2\mb{Z}\times\mb{Z}/2\mb{Z}$.)
\end{enumerate}
\end{thm}

\begin{rem}
When $x\in U_\ell$ is a rational intersection point of $f_1,\ldots,f_n$, then the local degree $\deg_x(f_1,\ldots,f_n)=\langle J_\ell(x)\rangle$, where $J_\ell(x)\in\mb{Q}$ is the signed volume of the parallelpiped spanned by the gradients of $f_1,\ldots,f_n$ at $x$. Our previous discussion allows us to compute $\partial\langle J_\ell(x)\rangle$, so it remains to consider the case when $x$ is a non-rational intersection point of $f_1,\ldots,f_n$. When $J_\ell(x)$ is a square in the residue field $\mb{Q}(x)$ of $x$, then $\deg_x(f_1,\ldots,f_n)=\Tr_{\mb{Q}(x)/\mb{Q}}\langle 1\rangle$ is the {\it trace form} of the field extension $\mb{Q}(x)/\mb{Q}$. Trace forms of algebraic number fields have been studied extensively. Bayer-Fluckiger and Lenstra~\cite{BL90} showed that if $\mb{Q}(x)/\mb{Q}$ is an odd degree field extension, then  $\Tr_{\mb{Q}(x)/\mb{Q}}\langle 1\rangle=[\mb{Q}(x):\mb{Q}]\cdot\langle 1\rangle$. If $\mb{Q}(x)/\mb{Q}$ is an even degree extension and no Sylow 2-subgroups of $\Gal(\mb{Q}(x)/\mb{Q})$ are metacyclic, then one can use the Knebusch exact sequence of Witt rings to show that $\Tr_{\mb{Q}(x)/\mb{Q}}\langle 1\rangle=[\mb{Q}(x):\mb{Q}]\cdot\langle 1\rangle$ if $\mb{Q}(x)$ is totally real and $\Tr_{\mb{Q}(x)/\mb{Q}}\langle 1\rangle=\frac{[\mb{Q}(x):\mb{Q}]}{2}\cdot\mb{H}$ if $\mb{Q}(x)$ is totally imaginary~\cite{CNC17}. When $J_\ell(x)$ is not a square, we remark that the discriminant can be computed by
\[\disc\Tr_{\mb{Q}(x)/\mb{Q}}\langle J_\ell(x)\rangle=\norm(J_\ell(x))\cdot D,\]
where $D=\disc\Tr_{\mb{Q}(x)/\mb{Q}}\langle 1\rangle$ is the discriminant (up to squares) of the number field $\mb{Q}(x)/\mb{Q}$.
\end{rem}

As an application of Theorem~\ref{thm:over Q}, we discuss intersections of a line and a conic in $\mb{P}^2_\mb{Q}$.

\begin{ex}\label{ex:over Q}
Let $f$ be a line and $g$ be a conic in $\mb{P}^2_\mb{Q}$. If $f$ and $g$ intersect with multiplicity 2 at a rational point $s$, then $\deg_s(f,g)=\mb{H}$. If $f$ and $g$ intersect at a non-rational point $s$, then $i_s(f,g)\geq[\mb{Q}(s):\mb{Q}]>1$ by Proposition~\ref{prop:rank=mult}. Hence $i_s(f,g)=2$ by the classical version of B\'ezout's Theorem, so $s$ must have a quadratic residue field. Thus $f$ and $g$ intersect transversely at $s$ by Proposition~\ref{prop:transverse=rank1}, and we have $\deg_s(f,g)=\Tr_{\mb{Q}(s)/\mb{Q}}\langle J(s)\rangle=\mb{H}$. This restricts the possible values of $J(s)$. For example, since $\disc\mb{H}=-1$, the Hasse--Minkowski principle implies that the discriminant of $\Tr_{\mb{Q}(s)/\mb{Q}}\langle J(s)\rangle$ must also be equal to $-1$ up to squares. If $D$ is the field discriminant (up to squares) of $\mb{Q}(s)/\mb{Q}$, then we have $\mb{Q}(s)\cong\mb{Q}(\sqrt{D})$. We may thus write $J(s)=a+b\sqrt{D}$, and we have $\disc\Tr_{\mb{Q}(s)/\mb{Q}}\langle J(s)\rangle=4D(a^2-b^2D)=D(a^2-b^2D)$ in $\mb{Q}^\times/(\mb{Q}^\times)^2$. This implies that, up to squares, we have $D(a^2-b^2D)+1=0$, so there is a forced relationship between $J(s)$ and the residue field of $s$. If $\mb{Q}(s)\cong\mb{Q}(i)$, for example, then we have $a^2+b^2=1$ up to squares in $\mb{Q}^\times$, so $a^2+b^2=\norm(J(s))$ must be a square in $\mb{Q}^\times$.

Now assume that $f$ and $g$ intersect at two distinct points $s,t$. By B\'ezout's Theorem, we know that $i_s(f,g)=i_t(f,g)=1$, so $f$ and $g$ intersect transversely at each of these $\mb{Q}$-rational points. Let $J(s)$ (respectively $J(t)$) denote the area of the parallelogram determined by the normal vectors of $f$ and $g$ at $s$ (respectively $t$). Theorem~\ref{thm:over Q} places various restrictions on the possible values of $J(s)$ and $J(t)$. In particular, $J(s)$ and $J(t)$ must have opposite signs and their dyadic valuations must agree mod 2. The local residues of $\langle J(s)\rangle$ and $\langle J(t)\rangle$ at odd primes also constrain the possible intersection types of $f$ and $g$. For example, it is impossible to have $J(s)$ be a non-square integer (other than $-1$) and $J(t)$ any integer prime to $J(s)$. Indeed, assume that $J(s)\neq -1$ is a non-square integer and $J(t)$ is an integer prime to $J(s)$, and let $p$ be a prime dividing $J(s)$ such that $v_p(J(s))$ is odd. Since $J(s)$ and $J(t)$ are coprime, we have that $p\nmid J(t)$ and hence $\partial_p\langle J(s)\rangle=\langle a\rangle$ for some $a\in\mb{F}_p^\times$ and $\partial_p\langle J(t)\rangle=0$. Thus $\partial_p\langle J(s)\rangle+\partial_p\langle J(t)\rangle$ is not trivial in $\W(\mb{F}_p)$, contradicting Theorem~\ref{thm:over Q}.
\end{ex}

\bibliography{bezout}{}
\bibliographystyle{alpha}
\end{document}